\documentclass{article}
\usepackage[utf8]{inputenc}
\usepackage[T1]{fontenc}
\usepackage{url}

\usepackage[pdftex,breaklinks,pdfpagemode={None},pdfstartview={FitH},
            pdfview={FitH},colorlinks,linkcolor={black},
            citecolor={black}, urlcolor={black}]{hyperref}
\usepackage{graphicx} 
\usepackage{amsthm}
\usepackage{amssymb,amsmath,color}
\usepackage{bbm}
\usepackage{MnSymbol}
\newtheorem{theorem}{Theorem}[section]

\newtheorem{corollary}[theorem]{Corollary}
\newtheorem{lemma}[theorem]{Lemma}
\newtheorem{defin}[theorem]{Definition}

\usepackage{tikz}
\usetikzlibrary{graphs,decorations.pathmorphing, decorations.pathreplacing, angles, quotes, shapes, arrows, topaths, calc, patterns, backgrounds, positioning, arrows, shapes.multipart}

\tikzset{st/.style = {circle,draw = blue!40,fill = blue!40,
  inner sep = 1pt, outer sep = 0pt}}
\tikzset{stt/.style = {circle,draw = teal!100,fill = teal!100,
  inner sep = 1.5pt, outer sep = 0pt}}  
\tikzset{inn/.style = {circle,draw = black!20,fill = black!10,
  inner sep = 2pt, outer sep = 0pt}}
\tikzset{nwtext/.style = {draw = white,fill = white,
    anchor= north west,inner sep = 2pt, outer sep = 0pt}}
\tikzset{wtext/.style = {draw = white,fill = white,
    anchor= west,inner sep = 2pt, outer sep = 0pt}}
\tikzset{elip/.style={ellipse,draw=black,fill=blue!15,
   inner sep = 2pt, outer sep = 2pt}}
\tikzset{min/.style={inner sep = 0pt, outer sep = 0pt}}

\newcommand{\tn}{\tilde{n}}
\newcommand{\hh}[1]{\frac{#1}{2}}
\newcommand{\hf}[1]{\lfloor\frac{#1}{2}\rfloor}
\newcommand{\hc}[1]{\lceil\frac{#1}{2}\rceil}
\newcommand{\tr}{\textrm}

\newcommand{\email}[1]{\href{mailto:#1}{\tt#1}}

\date{}
\begin{document}
\title{Weighted Cages}

\author{G. Araujo-Pardo$^{1,4}$ \and C. De la Cruz$^{2,5,6}$ \and
  M. Matamala$^{3,7}$ \and M. A. Pizaña$^{2,6}$}

\maketitle

\footnotetext[1]{Universidad Nacional Autónoma de México. \email{garaujo@im.unam.mx}}
\footnotetext[2]{Universidad Autónoma Metropolitana. \href{mailto:clau_mar@ciencias.unam.mx}{\tt clau\_mar@ciencias.unam.mx}, \email{mpizana@gmail.com}}
\footnotetext[3]{Department of Mathematical Engineering and Center for Mathematical Modeling, CNRS IRL 2807, Universidad de Chile, Santiago, Chile. \email{mar.mat.vas@dim.uchile.cl}}
\footnotetext[4]{Partially supported by PAPIIT-M\'exico under grant IN113324 and CONAHCyT: CBF 2023-2024-552.}
\footnotetext[5]{Partially supported by CONAHCYT, grant 799875.}
\footnotetext[6]{Partially supported by CONAHCYT, grant A1-S-45528.}
\footnotetext[7]{Supported by Centro de Modelamiento Matemático (CMM), FB210005, BASAL funds for centers of excellence from ANID-Chile.}
\renewcommand{\thefootnote}{\arabic{footnote}}

\begin{abstract}Cages ($r$-regular graphs of girth $g$ and minimum
  order) and their variants have been studied for over seventy
  years. Here we propose a new variant, \emph{weighted cages}. We
  characterize their existence; for cases $g=3,4$ we determine their
  order; we give Moore-like bounds and present some computational
  results.\end{abstract}

\section{Introduction}\label{sec:intro}

Cages~\cite{EJ13} have been studied since 1947 when they were
introduced by Tutte in \cite{Tut47}.  They are regular graphs of a
given girth with the smallest number of vertices for the given
parameters.  In 1963 Sachs \cite{Sac63} proved that for each $k\geq 2$
and each $g\geq 3$ there is a $k$-regular graph of girth $g$ which
implies that a cage exists for each such parameters. The smallest
integer $n$ for which there is a $k$-regular graph of girth $g$ on $n$
vertices is denoted by $n(k,g)$ and a $k$-regular graph of girth $g$
with $n(k,g)$ vertices is called a $(k,g)$-cage. Several variations of
the notion of cage have been studied in the literature including, among
others: Biregular cages~\cite{AABLL13,EJ16}, biregular bipartite
cages~\cite{AJRS22,FRJ19}, vertex-transitive cages~\cite{JS11}, Cayley
cages~\cite{EJS13}, mixed cages~\cite{ADG22,AHM19,Exo25} and mixed
geodetic cages~\cite{TE22}.

Standard terminology on graph theory used here, will be quickly
reviewed in the next section.

In this work, we extend the notion of cage to weighted graphs. In
general, a \emph{weighted graph}, is a (simple, finite) graph
$G=(V,E)$ together with a weight function
$w:E(G)\rightarrow \mathbb{R}$, however, to keep the presentation as
simple as possible, we shall focus on weight functions of the form
$w:E(G)\rightarrow \{1,2\}$. An edge with weight 1 is called a
\emph{light edge} while one with weight 2 is called a \emph{heavy
  edge}.

Under these circumstances, each weighted graph $G$, \emph{wgraph} for
short, may be specified by a couple of graphs $L=L(G)$ and $H=H(G)$,
which are the spanning subgraphs of $G$ formed by the light edges and
the heavy edges of $G$ (respectively). Hence, we shall represent a
wgraph $G$ by $G=(L,H)$, where $L$ and $H$ are graphs such that
$V(L)=V(H)$ and $E(L)\cap E(H)=\varnothing$.

In order to maintain the regularity aspect of the original notion we
require $L$ and $H$ to be regular. Thus an \emph{$(a,b)$-regular
  wgraph} is a wgraph $G=(L,H)$ where $L$ is $a$-regular and $H$ is
$b$-regular. A \emph{wcycle} in $G$ is a cycle whose edges may be
light or heavy and its \emph{weight} is the sum of the weights of the
edges composing it. The \emph{girth} of a wgraph $G$ is the minimum
weight of its wcycles. Finally, by analogy with cages, we may define
an \emph{$(a,b,g)$-wgraph} as an $(a,b)$-regular wgraph of girth $g$,
and an \emph{$(a,b,g)$-wcage} as an \emph{$(a,b,g)$-wgraph} of minimum
order.  We shall represent the order of an $(a,b,g)$-wcage by
$n(a,b,g)$.

In this paper, we characterize their existence and, for the cases
$g=3,4$, we determine the value of $n(a,b,g)$; We also determine
$n(a,b,g)$ for $a=1,2$ when $g=5,6$. We give Moore-like bounds and
present some computational results.  

An interesting feature of weighted cages is that, contrary to what
happens with ordinary cages, $n(a,b,g)$ is not always monotone
increasing in all its parameters, since we shall see that
$n(3,1,4) = 8 > 6 = n(3,2,4)$ in Section~\ref{sec:g4} and that
$n(4,1,5) = 20 > 19 = n(4,2,5)$ in Section~\ref{sec:exp}.

We note that many of our results may be readily extended to weights of
the form $w:E(G)\rightarrow \{w_{1},w_{2}\}\subset\mathbb{N}$.

\section{Terminology and Preliminaries}\label{sec:prelim}
 
Our graphs are simple and finite. We use standard terminology for
denoting the \emph{set of vertices} and the \emph{set of edges} of a
graph $X$: $X=(V,E)$, $V=V(X)$ and $E=E(X)$. The \emph{order} of a
graph $X$ is $|X|=|V(X)|$. An \emph{edge} is an unordered pair of
vertices $\{x,y\}$, which we may also write as $xy$. We write
$x\simeq_{X} y$ for the \emph{adjacent-or-equal relation} on a graph
$X$. The \emph{degree} of a vertex $x$ in $X$ is defined by
$deg_{X}(x)=|\{xy : xy \in E(X)\}|$. The maximum degree is
$\Delta(X)=\max \{deg_{X}(x) : x \in V(X)\}$. A graph $X$ is
$r$-regular if $deg(x)=r$, for all vertices $x\in V(X)$. The
\emph{distance} between vertices $x$ and $y$ in $X$ is denoted by
$dist_X(x,y)$. The \emph{complete graphs} on $n$ vertices are
represented by $K_{n}$ and the \emph{complete balanced bipartite
  graphs} on $n$ vertices are denoted by $K_{m,m}$, where
$m=\hh{n}$. Given graphs $X$ and $Y$, some standard operations on
graphs are: the \emph{complement} of a graph
$\overline{X} = (V(X),\overline{E(X)})$, where
$\overline{E(X)}=\{\{x,y\} : x,y\in V(X), x\neq y\textup{ and }
\{x,y\}\not\in E(X)\}$, the \emph{square} of a graph
$X^{2}=(V(X),E(X^{2}))$, where
$E(X^{2})=\{\{x,y\} : 0< dist_{X}(x,y)\leq 2\}$ and the \emph{union}
of graphs $X\cup Y = (V(X)\cup V(Y), E(X)\cup E(Y))$, while the
\emph{disjoint union} is
$X\cupdot Y = (V(X)\cupdot V(Y),E(X)\cupdot E(Y))$. Here we define the
\emph{difference} of graphs as $X-Y=(V(X),E(X)-E(Y)$, that is, the
edges of $Y$, are removed from $X$, but not the vertices. The
\emph{girth} $g(X)$ of a graph, $X$, is the length of a shortest
cycle in $X$. An \emph{$(r,g)$-graph} is an $r$-regular graph of girth
$g$. An \emph{$(r,g)$-cage} is an $(r,g)$-graph of minimum order. The
order of an $(r,g)$-cage is denoted by $n(r,g)$; when no such cage
exists, we define $n(r,g)=\infty$ (this happens exactly when $r< 2$
or $g<3$).

A \emph{weighted graph} (\emph{wgraph} for short) is $G=(L,H)$, where
$L=L(G)$ is the \emph{light-subgraph} of $G$ and $H=H(G)$ is the
\emph{heavy-subgraph} of $G$; both $L$ and $H$ are ordinary graphs and
we require that $V(L)=V(H)$ and $E(L)\cap E(H)=\varnothing$. Light
edges have weight 1 and heavy edges have weight 2.  A \emph{wcycle}
(\emph{wpath}) in $G$ is a cycle (path) whose edges may be light or
heavy and its \emph{weight} is the sum of the weights of the edges
composing it. The \emph{wdistance} between two vertices $x$ and $y$ in
$G$ is the minimum weight of a wpath in $G$ joining $x$ and $y$. Other
terms like \emph{wtree} and \emph{subwgraph} will be used with the
obvious meaning.

We say that $G=(L,H)$ is \emph{$(a,b)$-regular} if $L$ is $a$-regular
and $H$ is $b$-regular. The \emph{girth}, $g(G)$, of a wgraph is the
minimum weight of a wcycle in $G$.  An \emph{$(a,b,g)$-wgraph} $G$ is
an $(a,b)$-regular wgraph $G$ of girth $g$ and an
\emph{$(a,b,g)$-wcage} is an $(a,b,g)$-wgraph of minimum order. We
define $n(a,b,g)$ as the order of an $(a,b,g)$-wcage (and we define
$n(a,b,g)=\infty$ if there is no such $(a,b,g)$-wcage).

It should be clear that $n(a,b,g)=\infty$ whenever $a+b\leq 1$ or
$g<3$.  Also, it is immediate that $n(a,0,g)=n(a,g)$ and that
$n(0,b,g)=n(b,\hh{g})$ whenever $g$ is even and $g\geq 6$ (otherwise,
$n(0,b,g)=\infty$). A $(1,1,g)$-wcage must be a wcycle of weight $g$
with alternating light and heavy edges, and hence:
$$n(1,1,g)=\begin{cases} 
\frac{2g}{3} & \textup{if } g\geq 6\textup{ and }g\equiv 0 \mod 3,\\ 
\infty & \textup{otherwise.}
\end{cases}$$

We shall use \emph{congruence module 2} very often, and hence we shall
abbreviate ``$x \equiv y \mod 2$'' simply as ``$x\equiv y$''. It is a
well know result (sometimes called \emph{the first theorem of graph
  theory} or the \emph{degree-sum formula}) that the sum of the
degrees of a graph is even (and equals twice the number of edges). For an $r$-regular graph of order $n$, this means $nr\equiv 0$, and hence
that there are no odd-regular graphs of odd order. This fact will be
used very often in this paper and we shall refer to it simply as
``\emph{parity forbids}'', as in: ``parity forbids $r=3$ and
$n=7$''. We shall often need the following four lemmas:

\begin{lemma}\label{lem:trivialbound}
  Let $a,b\geq 0$ and $g\geq 3$.  Then $n(a,b,g)\geq a+b+1$. Moreover,
  if $ab \equiv 1$, then $n(a,b,g)\geq a+b+2$.
\end{lemma}
\begin{proof}
  If there is no $(a,b,g)$-wcage, then, by definition,
  $n(a,b,g)=\infty$ and the inequalities hold. Otherwise, take an
  $(a,b,g)$-wcage $G=(L,H)$ and a vertex $x\in G$. Then $x$ must have
  $a$ neighbors in $L$ and $b$ neighbors in $H$, and therefore the
  closed neighborhood of $x$ in $G$ must have $a+b+1$ vertices. Thus
  $n(a,b,g)=|G|\geq a+b+1$. Parity forbids $n=a+b+1$ when
  $ab\equiv 1$, hence $n(a,b,g)\geq a+b+2$ in that case.
\end{proof}

Recall that a \emph{$k$-factor}, $F$, of a graph $X$ is a $k$-regular
spanning subgraph of $X$. Thus a $1$-factor is a \emph{perfect matching} and a
$2$-factor is a collection of cycles that span all of $X$. A
\emph{$k$-factorization} of $X$ is a decomposition of $X$ into
$k$-factors, that is, a collection of $k$-factors $\{F_i\}_{i\in I}$,
such that $E(F_{i})\cap E(F_{j})=\varnothing$ for all $i\neq j$ and
$G = \bigcup_{i\in I} F_{i}$.

\begin{lemma}\label{lem:2factor}
  If $5 \leq n\equiv 1$, there is a $2$-factorization of $K_{n}$,
  $K_{n}=\bigcup_{i=1}^{\hf{n}} F_{i}$, such that $F_{1}\cup F_{2}$
  contains a triangle.
\end{lemma}
\begin{proof} Label the vertices of $K_{n}$ with the elements of
  $\mathbb{Z}_{n}$. For $i\in \{1,2,\ldots,\hf{n}\}$ define $F_{i}$ as
  the spanning subgraph of $K_{n}$ having edge set
  $E(F_{i})=\{\{x,x+i\} : x\in \mathbb{Z}_n\}$. It is straightforward
  to verify that $\{F_{i}\}_{i=1}^{\hf{n}}$ is a $2$-factorization of
  $K_{n}$. A triangle in $F_{1}\cup F_{2}$ is induced by the vertices
  $\{0,1,2\}$.
\end{proof}

\begin{lemma}\label{lem:1factor}
  If $4 \leq n\equiv 0$, there is a $1$-factorization of
  $K_{n}$, $K_{n}=\bigcup_{i=0}^{n-2} \tilde{F}_i$, such that
  $\tilde{F}_{0}\cup \tilde{F}_{1}\cup \tilde{F}_{2}$ contains a
  triangle.
\end{lemma}
\begin{proof}
  Label one vertex of $K_{n}$ as $\ast$ and label the remaining
  vertices with the elements of $\mathbb{Z}_{n-1}$. For
  $i\in \mathbb{Z}_{n-1}$, define $\tilde{F}_{i}$ as the spanning
  subgraph of $K_{n}$ having edge set
  $E(\tilde{F}_{i})=\{\{\ast,i\}\}\cup \{\{i+k,i-k\} : k \in
  \{1,2,\ldots,\hh{n-2}\}\}$. It is straightforward to verify that
  $\{\tilde{F}_{i}\}_{i=0}^{n-2}$ is a $1$-factorization of $K_{n}$. A
  triangle in $\tilde{F}_{0}\cup \tilde{F}_{1}\cup \tilde{F}_{2}$ is
  induced by the vertices $\{\ast,0,2\}$.
\end{proof}

\begin{lemma}\label{lem:1factorBip}
  If $m\geq 3$ there is a $1$-factorization of $K_{m,m}$,
  $K_{m,m}=\bigcup_{i=0}^{m-1} \hat{F}_{i}$, such that
  $\hat{F}_{0}\cup \hat{F}_{1}\cup \hat{F}_{2}$
  contains a 4-cycle.
\end{lemma}
\begin{proof}
  Let $\{X,Y\}$ be the bipartition of $K_{m,m}$. Label the vertices of
  $X$ with $\{x_{i} : i\in \mathbb{Z}_{m} \}$ and the vertices of $Y$
  with $\{y_{i} : i \in \mathbb{Z}_{m}\}$. Define $\hat{F}_i$ as
  the spanning subgraph of $K_{m,m}$ with edge set
  $E(\hat{F_{i}})=\{x_{j}y_{j+i}: j \in \mathbb{Z}_m\}$. It is
  straightforward to verify that $\{\hat{F}_{i}\}_{i=0}^{m-1}$ is a
  $1$-factorization of $K_{m,m}$. A $4$-cycle is induced in
  $\hat{F}_{0}\cup \hat{F}_{1}\cup \hat{F}_{2}$ by the
  vertices $\{x_{0},y_{1},x_{1},y_{2}\}$.
\end{proof}

\section{Existence  of weighted cages}\label{sec:exist}
Given two graphs $Z$, $Y$, a \emph{(weak) morphism},
$\varphi:Z\rightarrow Y$, is a function
$\varphi:V(Z)\rightarrow V(Y)$, such that $z\simeq_{Z} z'$ implies
$\varphi(z)\simeq_{Y} \varphi(z')$. Note that $\varphi$ may map adjacent
vertices into equal vertices. For $zz'\in E(Z)$ we define
$\varphi(zz') = \{\varphi(z),\varphi(z')\}$ which may be singleton or
an edge in $Y$. We also define
$\varphi^{-1}(y)=\{z\in V(Z) : \varphi(z) = y\}$ and
$\varphi^{-1}(yy')=\{zz' \in E(Z): \varphi(zz') = yy'\}$.

Recall that a wcycle is a cycle composed by light and heavy edges and
that its weight is the sum of the weights of its edges.  An
$(a,b,g)$-wcycle is a wcycle $C=(L,H)$ of weight $g$ such that, for
each $x\in V(C)$, $deg_{L}(x)\leq a$ and $deg_{H}(x)\leq b$. Hence, if
$C$ is an $(a,b,g)$-wcycle, then it is also an $(a',b',g)$-wcycle,
whenever $a'\geq a$ and $b'\geq b$. For instance, a cycle of length
$g$ composed only of light edges is an $(a,0,g)$-wcycle, for each
$a\geq 2$. Similarly, a cycle of length $\ell$ composed only of heavy
edges is a $(0,b,2\ell)$-wcycle, for each $b\geq 2$. Also, any wcycle
of weight $g$ is a $(2,2,g)$-wcycle, but there are no
$(a,b,g)$-wcycles when $a+b\leq 1$. Note that any $(a,b,g)$-wcage
contains at least one $(a,b,g)$-wcycle. 

We shall prove in this section that an $(a,b,g)$-wcage exists whenever an
$(a,b,g)$-wcycle exists.  The idea is very simple: Start by taking
such an $(a,b,g)$-wcycle, extend it to achieve the light-regularity and
then extend it again to achieve the heavy-regularity. The formal
details, however, require a series of lemmas. Let us begin by
characterizing the existence of $(a,b,g)$-wcycles:

\begin{lemma}\label{existence} Let $a,b\geq 0$ and $g\geq 3$, then
 an $(a,b,g)$-wcycle exists if and only if any of the conditions 1-4 holds: 
\begin{enumerate}
\item[1.] $a\geq 2$.
\item[2.] $a=1$, $b\geq 2$, and $g \geq 5$.
\item[3.] $a=1$, $b=1$, $g\geq 6$ and $g\equiv 0 \mod 3$.
\item[4.] $a=0$, $b\geq 2$, $g\geq 6$ and $g\equiv 0 \mod 2$.
\end{enumerate}
\end{lemma}
\begin{proof}
  Case 1: A wcycle can be formed using only light edges.  Case 2: A
  wcycle can be formed either using only heavy edges (for even $g$,
  with $g\geq 6$) or using one light edge and $(g-1)/2$ heavy edges
  (for odd $g$, with $g\geq 5$).  Case 3: Any such wcycle must
  alternate light and heavy edges; any two such consecutive edges in
  the wcycle contribute 3 to the weight of the wcycle and hence
  $g\equiv 0 \mod 3$. Also, the minimum of such wcycles has 4 edges
  and $g=6$.  Case 4: Any wcycle must contain only heavy edges and
  hence $g\equiv 0 \mod 2$. Also the minimum of such wcycles has 3 
  edges and $g=6$.

 It is straightforward to verify that these are all the cases in which an $(a,b,g)$-wcycle exists. 
\end{proof}

\begin{defin}\label{semidirect}
Given graphs $X$ and $Y$ we say that $Z$ is a \emph{semidirect product} of
$X$ and $Y$ (written as $Z=X\rtimes Y$ or $Z=X\rtimes_{\varphi} Y$) whenever:
\begin{enumerate}
\item There is a morphism $\varphi:Z\rightarrow Y$
\item $\varphi^{-1}(y)\cong X$, for every $y\in V(Y)$.
\item $|\varphi^{-1}(y_{1}y_{2})|=1$, for every $y_{1}y_{2}\in E(Y)$.
\end{enumerate}
\end{defin}

Note that, given $Z=X\rtimes_{\varphi} Y$, we must have that $\varphi$
is vertex- and edge-surjective, that $|Z|=|X||Y|$, and that $Z$ is the
disjoint union of $|Y|$ copies of $X$ with some additional
\emph{external edges}, which only connect vertices from different
copies of $X$ in $Z$ and such that given two such copies $X_{1}$ and
$X_{2}$ of $X$ in $Z$, there is at most one external edge connecting a
vertex from $X_{1}$ to a vertex from $X_{2}$.

\begin{lemma}[Extension Lemma]\label{lem:ext}
Let $d\geq 0$ be an integer. Let $X$ be a graph with $\Delta(X)\leq d$.  Define
the \emph{defect} $D = d\cdot|X|-\sum_{x \in X} deg_X(x)$.  Let $Y$ be
a $D$-regular graph. Then there is a $d$-regular graph $Z$ with 
$Z=X\rtimes Y$. 
\end{lemma}
\begin{proof} Let us construct $Z$ and $\varphi$ as follows. First
  take $V(Z)=V(X)\times V(Y)$. Define $\varphi: Z\rightarrow Y$ by
  $\varphi(x,y)=y$. Add the following edges to $Z$:
  $$\{(x,y)(x',y') : xx'\in E(X) \textup{ and } y=y' \}.$$ At this point, we
  already have $\varphi^{-1}(y)\cong X$, for every $y\in V(Y)$.

  Given an edge $yy' \in E(Y)$ select any pair of vertices
  $z=(x,y)\in \varphi^{-1}(y)$ and $z'=(x',y') \in \varphi^{-1}(y' )$
  satisfying $deg_{Z}(z)<d$ and $deg_{Z}(z') < d$ (if any). If the
  selection was possible, add the edge $zz'$ to $Z$ and mark
  the edge $yy'$ as \emph{used}. Repeat this procedure with the rest
  of the \emph{unused} edges of $Y$ until it is impossible to add more
  edges to $Z$. In this way, we just added at most one edge to $Z$ for
  each edge of $Y$. Note that $deg_{Z}(z)\leq d$ for all $z\in Z$.  We
  claim that $Z$ already possesses all the required properties.

  Assume first that all edges of $Y$ were used. Then to each copy
  $\varphi^{-1}(y)$ of $X$ (for any $y\in Y$), we just added
  $deg_{Y}(y)=D$ new \emph{external} edges (each ending in another
  copy of $X$). Then, recalling the definition of the defect $D$, the new degree sum of the all vertices
  $z=(x,y)$ in $\varphi^{-1}(y)$ is:
  \begin{equation}\label{degsum}
  \sum_{z\in \varphi^{-1}(y)} deg_{Z}(z) = \sum_{x \in X} deg_X(x) +  D = d\cdot|X|.\end{equation} Since $deg_{Z}(z)\leq d$ and $|\varphi^{-1}(y)|=|X|$, Equation~(\ref{degsum}) implies that $deg_{Z}(z)= d$, for all $z\in \varphi^{-1}(y)$. Since this happens for every $y$, it follows 
  that $Z$ is $d$-regular. It should be clear that all the conditions in
  Definition~\ref{semidirect} are satisfied.

  Finally, assume that some edge $yy'\in E(Y)$ could not be used. This
  means, without lost of generality, that every vertex
  $z\in\varphi^{-1}(y)$ already had degree $d$. But then
  $\sum_{z\in \varphi^{-1}(y)} deg_{Z}(z) =d\cdot|X| = \sum_{x \in X}
  deg_X(x)+D$, which mean that $D$ additional external edges were
  added during the procedure to the vertices of $\varphi^{-1}(y)$. But
  $deg_{Y}(y)=D$ and hence all edges incident with $y$ in $Y$ were
  used. Therefore $yy' $ was already used indeed. A contradiction.
\end{proof}

\begin{theorem}\label{construction}
  For $a,b\geq 0$, $g\geq 3$, an $(a,b,g)$-wcage exists if and only if
  an $(a,b,g)$-wcycle exists.
\end{theorem}
\begin{proof}
  It suffices to show that an $(a,b,g)$-wgraph exists. 
  Figure~\ref{fig:construction} illustrates this proof. 
  Let $G_{0}$ be an $(a,b,g)$-wcycle, so $g(G_{0})=g$. 

\begin{figure}[h]
\begin{center}
\begin{tikzpicture}

  \begin{scope}[scale=0.7]

   \node at (1.5,4) {$G_0$};  
   \node [st] (v1) at (1,2) {};
   \node [st] (v2) at (2,2) {};
   \node [st] (v3) at (1.5,3) {};
   \draw [thick,black] (v1)--(v2);
   \draw [very thick,red] (v1)--(v3)--(v2);

   \node at (5.5,4) {$X_0$};  
   \node [st] (v4) at (5,2) {};
   \node [st] (v5) at (6,2) {};
   \node [st] (v6) at (5.5,3) {};
   \draw [thick,black] (v4)--(v5);   
   
   \node at (9.5,4) {$Y_0$};     
   \node [stt] (v7) at (9,2.5) {};
   \node [stt] (v8) at (10,2.5) {};
   \draw [thick,black] (v7)--(v8);  
   \node at (9.5,1.5) {$D=1$};    
   \node at (9.5,1) {$g(Y_0)\geq 5$};   
  \end{scope}

\end{tikzpicture}\vspace{0.4cm}

\begin{tikzpicture}
  \begin{scope}[scale=0.7]

   \node at (1.5,4) {$Z_1$}; 
   \draw[teal!100] (0.5,2.5) circle (25pt);
   \node [st] (v1) at (0,2) {};
   \node [st] (v2) at (1,2) {};
   \node [st] (v3) at (0.5,3) {};
   \draw [thick,black] (v1)--(v2);
   \draw[teal!100] (2.5,2.5) circle (25pt);   
   \node [st] (v4) at (2,2) {};
   \node [st] (v5) at (3,2) {};
   \node [st] (v6) at (2.5,3) {};
   \draw [thick,black] (v4)--(v5);
   \draw [thick,black] (v3)--(v6);
   
   \node at (5.75,4) {$G_1$};   
   \node [st] (v7) at (4.5,2) {};
   \node [st] (v8) at (5.5,2) {};
   \node [st] (v9) at (5,3) {};
   \draw [very thick,red] (v7)--(v9)--(v8);   
   \draw [thick,black] (v7)--(v8);  
   \node [st] (v10) at (6,2) {};
   \node [st] (v11) at (7,2) {};
   \node [st] (v12) at (6.5,3) {};
   \draw [very thick,red] (v10)--(v12)--(v11);    
   \draw [thick,black] (v10)--(v11);
   \draw [thick,black] (v9)--(v12);  
   
   \node at (9.25,4) {$X_1$};     
   \node [st] (v13) at (8,2) {};
   \node [st] (v14) at (9,2) {};
   \node [st] (v15) at (8.5,3) {};
   \draw [very thick,red] (v13)--(v15)--(v14);     
   \node [st] (v16) at (9.5,2) {};
   \node [st] (v17) at (10.5,2) {};
   \node [st] (v18) at (10,3) {};
   \draw [very thick,red] (v16)--(v18)--(v17);   
   \node at (13,0.5) {$D=4$};    
   \node at (13,0) {$g(Y_1)\geq 3$};
   
   \node at (13,4) {$Y_1$};  
   \node [stt] (v19) at (13,3.31) {};   
   \node [stt] (v20) at (14.21,2.43) {}; 
   \node [stt] (v21) at (13.75,1) {}; 
   \node [stt] (v22) at (12.25,1) {};   
   \node [stt] (v23) at (11.79,2.43) {}; 
   \draw [thick,black] (v19)--(v20)--(v21)--(v22)--(v23)--(v19)--(v21)--(v23)--(v20)--(v22)--(v19);   
  \end{scope}
\end{tikzpicture}
\begin{tikzpicture}
\begin{scope}[scale=1.5]
   \node at (0.1,1.3) {$Z_2$}; 
   \node [st] (v1) at (1.25,2.55) {};
   \node [st] (v2) at (1.45,2.55) {};
   \node [st] (v3) at (1.35,2.75) {};
   \draw [very thick,red] (v1)--(v3)--(v2);     
   \node [st] (v4) at (1.55,2.55) {};
   \node [st] (v5) at (1.75,2.55) {};
   \node [st] (v6) at (1.65,2.75) {};
   \draw [very thick,red] (v4)--(v6)--(v5);  
   \draw[teal!100] (1.5,2.65) circle (12pt);   

   \node [st] (v7) at (2.24,2.19) {};
   \node [st] (v8) at (2.3,2) {};
   \node [st] (v9) at (2.46,2.15) {};
   \draw [very thick,red] (v7)--(v9)--(v8);     
   \node [st] (v10) at (2.33,1.9) {};
   \node [st] (v11) at (2.39,1.71) {};
   \node [st] (v12) at (2.55,1.87) {};
   \draw [very thick,red] (v10)--(v12)--(v11); 
   \draw[teal!100] (2.37,1.97) circle (12pt);   
   
   \node [st] (v13) at (2.2,1.14) {};
   \node [st] (v14) at (2.04,1.02) {};
   \node [st] (v15) at (2.23,0.92) {};
   \draw [very thick,red] (v13)--(v15)--(v14);     
   \node [st] (v16) at (1.95,0.96) {};
   \node [st] (v17) at (1.79,0.84) {};
   \node [st] (v18) at (1.99,0.74) {};
   \draw [very thick,red] (v16)--(v18)--(v17); 
   \draw[teal!100] (2.06,0.9) circle (12pt);   
   
   \node [st] (v19) at (1.19,0.85) {};
   \node [st] (v20) at (1.02,0.97) {};
   \node [st] (v21) at (0.99,0.75) {};
   \draw [very thick,red] (v19)--(v21)--(v20);     
   \node [st] (v22) at (0.94,1.03) {};
   \node [st] (v23) at (0.78,1.14) {};
   \node [st] (v24) at (0.74,0.92) {};
   \draw [very thick,red] (v22)--(v24)--(v23); 
   \draw[teal!100] (0.9,0.9) circle (12pt); 
      
   \node [st] (v25) at (0.6,1.72) {};
   \node [st] (v26) at (0.66,1.91) {};
   \node [st] (v27) at (0.44,1.88) {};
   \draw [very thick,red] (v25)--(v27)--(v26);     
   \node [st] (v28) at (0.69,2.01) {};
   \node [st] (v29) at (0.75,2.2) {};
   \node [st] (v30) at (0.53,2.17) {};
   \draw [very thick,red] (v28)--(v30)--(v29); 
   \draw[teal!100] (0.6,1.97) circle (12pt);          
   
   \draw [very thick,red] (v1)--(v29);
   \draw [very thick,red] (v2)--(v22);       
   \draw [very thick,red] (v4)--(v14);
   \draw [very thick,red] (v5)--(v7);    
   \draw [very thick,red] (v8)--(v28);
   \draw [very thick,red] (v10)--(v20); 
   \draw [very thick,red] (v11)--(v13);
   \draw [very thick,red] (v16)--(v26); 
   \draw [very thick,red] (v17)--(v19);
   \draw [very thick,red] (v23)--(v25); 

   \node at (4.2,1.3) {$G_2$};   
   \node [st] (v31) at (5.25,2.53) {};
   \node [st] (v32) at (5.45,2.53) {};
   \node [st] (v33) at (5.35,2.73) {};
   \draw [very thick,red] (v31)--(v33)--(v32);     
   \node [st] (v34) at (5.55,2.53) {};
   \node [st] (v35) at (5.75,2.53) {};
   \node [st] (v36) at (5.65,2.73) {};
   \draw [very thick,red] (v34)--(v36)--(v35);  

   \node [st] (v37) at (6.24,2.17) {};
   \node [st] (v38) at (6.3,1.97) {};
   \node [st] (v39) at (6.46,2.13) {};
   \draw [very thick,red] (v37)--(v39)--(v38);     
   \node [st] (v40) at (6.33,1.88) {};
   \node [st] (v41) at (6.4,1.69) {};
   \node [st] (v42) at (6.55,1.85) {};
   \draw [very thick,red] (v40)--(v42)--(v41); 
   
   \node [st] (v43) at (6.2,1.11) {};
   \node [st] (v44) at (6.04,1) {};
   \node [st] (v45) at (6.24,0.89) {};
   \draw [very thick,red] (v43)--(v45)--(v44);     
   \node [st] (v46) at (5.96,0.94) {};
   \node [st] (v47) at (5.8,0.82) {};
   \node [st] (v48) at (5.99,0.72) {};
   \draw [very thick,red] (v46)--(v48)--(v47); 
   
   \node [st] (v49) at (5.19,0.83) {};
   \node [st] (v50) at (5.03,0.94) {};
   \node [st] (v51) at (4.99,0.72) {};
   \draw [very thick,red] (v49)--(v51)--(v50);     
   \node [st] (v52) at (4.95,1) {};
   \node [st] (v53) at (4.78,1.12) {};
   \node [st] (v54) at (4.75,0.9) {};
   \draw [very thick,red] (v52)--(v54)--(v53); 
   
   \node [st] (v55) at (4.6,1.7) {};
   \node [st] (v56) at (4.67,1.89) {};
   \node [st] (v57) at (4.44,1.86) {};
   \draw [very thick,red] (v55)--(v57)--(v56);     
   \node [st] (v58) at (4.7,1.99) {};
   \node [st] (v59) at (4.76,2.18) {};
   \node [st] (v60) at (4.54,2.14) {};
   \draw [very thick,red] (v58)--(v60)--(v59);   
   
   \draw [very thick,red] (v31)--(v59);
   \draw [very thick,red] (v32)--(v52);       
   \draw [very thick,red] (v34)--(v44);
   \draw [very thick,red] (v35)--(v37);    
   \draw [very thick,red] (v38)--(v58);
   \draw [very thick,red] (v40)--(v50); 
   \draw [very thick,red] (v41)--(v43);
   \draw [very thick,red] (v46)--(v56); 
   \draw [very thick,red] (v47)--(v49);
   \draw [very thick,red] (v53)--(v55);   
   
   \draw [thick,black] (v31)--(v32);
   \draw [thick,black] (v34)--(v35);
   \draw [thick,black] (v33)--(v36);
   \draw [thick,black] (v37)--(v38);
   \draw [thick,black] (v40)--(v41);
   \draw [thick,black] (v39)--(v42);   
   \draw [thick,black] (v43)--(v44);
   \draw [thick,black] (v46)--(v47);
   \draw [thick,black] (v45)--(v48);
   \draw [thick,black] (v49)--(v50);
   \draw [thick,black] (v52)--(v53);
   \draw [thick,black] (v51)--(v54);
   \draw [thick,black] (v55)--(v56);
   \draw [thick,black] (v58)--(v59);
   \draw [thick,black] (v57)--(v60);
   \end{scope}
\end{tikzpicture}
\end{center}
    \caption{Construction in the proof of Theorem~\ref{construction}.}
    \label{fig:construction}
\end{figure}
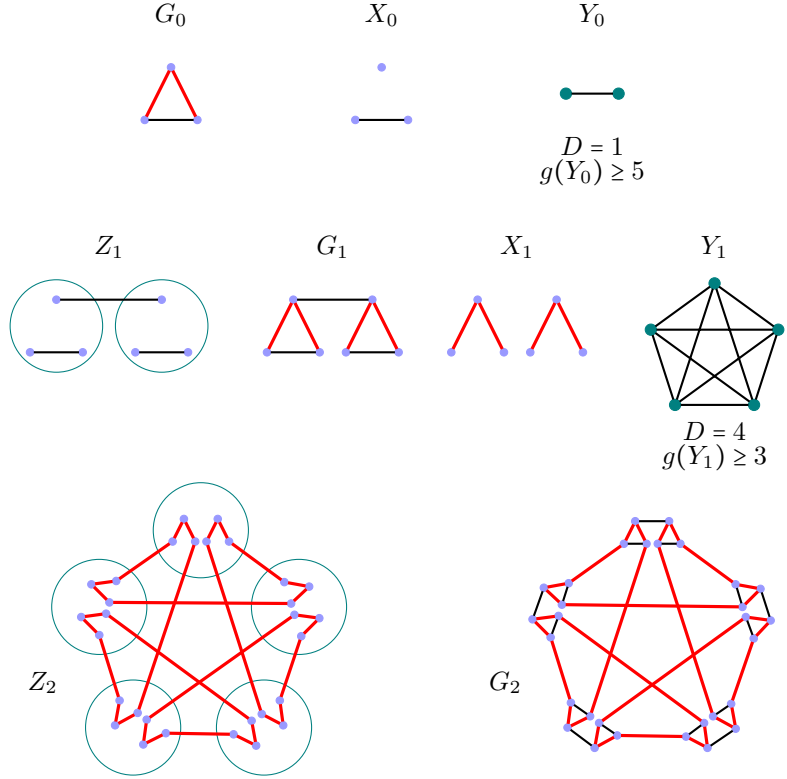

Let $X_{0}=L(G_{0})$, the light-subgraph of $G_0$. Note that $g(X_{0})\geq g$
and $\Delta(X_{0})\leq a$. Take $d=a$ and compute the defect
$D = d\cdot|X_{0}|-\sum_{x \in X_{0}} deg_{X_{0}}(x)$ as in the
Extension Lemma (\ref{lem:ext}). Now, select any $D$-regular graph
$Y_{0}$ of girth $g(Y_{0})\geq g$, for instance: for $D=0$ we may take
$Y_{0}=K_{1}$; for $D=1$ take $Y_{0}=K_{2}$; and for $D\geq 2$, we may
take $Y_{0}$ as any $(D,g)$-cage.

Now we use the Extension Lemma with $d=a$, to get an $a$-regular graph
$Z_{1}=X_{0}\rtimes_{\varphi}Y_{0}$ for some
$\varphi:Z_{1}\rightarrow Y_{0}$. Recall that $Z_{1}$ is the disjoint
union of $|Y_{0}|$ copies of $X_{0}$, with some additional
external edges. Now, in each of these copies of $X_{0}$ in $Z_{1}$ put
back the heavy edges originally present in $G_0$ (if any) to obtain
$G_{1}$ (i.e. $Z_{1}$ is a spanning subgraph of $G_{1}$).

We claim that $g(G_1)=g$. First note that the original wcycle $G_{0}$
is present in $G_{1}$, indeed, each copy of $X_{0}$ in $Z_{1}$ induce
a copy of $G_{0}$ in $G_{1}$. Hence $g(G_{1})\leq g$. But, if there
was a wcycle $C$ in $G_{1}$ of weight $g' < g$, then, this wcycle
must use external edges of $Z_{1}$ and, since $Z_{1}=X_{0}\rtimes_{\varphi}Y_{0}$ (see
Definition~\ref{semidirect}), $\varphi(C)$ must be a closed walk in
$Y_{0}$ which contain a wcycle $C'$ in $Y_{0}$ of weight
$g(C')\leq g(C) =g' < g$ which implies $g(Y_{0})<g$, a contradiction. It follows that $g(G_{1})=g$.

We now repeat the extension procedure for the heavy edges of $G_{1}$ to attain the desired heavy regularity:  

Let $X_{1}=H(G_{1})$, the heavy-subgraph of $G_{1}$. Take $d=b$ and compute the
defect $D = d\cdot|X_{1}|-\sum_{x \in X_{1}} deg_{X_{1}}(x)$.  Select
any $D$-regular graph $Y_{1}$ of girth
$g(Y_{1})\geq \hc{g}$. Note that this time
$g(Y_{1})\geq \hc{g}$ is enough since these edges are
going to be the heavy edges of the final graph. Now use the Extension
Lemma to get a $b$-regular graph $Z_{2}=X_{1}\rtimes Y_{1}$. In each
copy of $X_{1}$ in $Z_{2}$ put back the light edges originally present
in $G_{1}$ to obtain $G_{2}$ (i.e. $Z_{2}$ is a spanning subgraph of
$G_{2}$). It should be clear as before, that $g(G_{2})=g$ and that
$G_{2}$ is $(a,b)$-regular.
\end{proof}

The general construction in the previous theorem gives bad
general upper bounds. In several cases, we can get better
upper bounds using the same ideas as shown in the
Theorem~\ref{generalbounds}. 

The order of an $(r,g)$-cage, $n(r,g)$, is finite for $r\geq 2$ and
$g\geq 3$, but for the constructions used in
Theorem~\ref{generalbounds} we shall need this variant, $\tn(r,g)$, of
$n(r,g)$ which is finite for $r\geq 0$, $g \geq 2$:
$$\tn(r,g)=\begin{cases}
n(r,g) & \textup{if } r\geq 2, g\geq 3,\\
r+1 & \textup{if } 0\leq r \leq 1\textup{ or }g = 2.
\end{cases}$$

This function, $\tn(r,g)$, is the order of the smallest $r$-regular
graph $X$ of girth $g(X)\geq g$. It coincides with $n(r,g)$ when an
$(r,g)$-cage exists (i.e. when $r\geq 2$ and $g\geq 3$), otherwise $X$
is a complete graph on $r+1$ vertices and the girth of $X$ is either
$\infty$ (no cycles) or $3$.

\begin{theorem}\label{generalbounds}
In the indicated cases, the following upper bounds hold.\medskip

\begin{tabular}{ll}
1. $n(a,b,g) \leq n(a,g)\cdot \tn(b\cdot n(a,g),\hc{g})$ & for $a\geq 2$, $g\geq 3$.\\
2. $n(a,b,g) \leq n(b,\hh{g})\cdot \tn(a\cdot n(b,\hh{g}),g)$ & for $b\geq 2$, $g\geq 6$ , $g$ even. \\
3. $n(a,b,g) \leq 2\cdot n(a,g)$ & for $a \geq 2$, $b=1$, $g\leq 6$.\\ 
4. $n(a,b,g) \leq 2\cdot \tn(b,3)$ & for $a=1$, $b \geq 1$, $g= 6$. 
\end{tabular}
\end{theorem}
\begin{proof}
We shall use the Extension Lemma~\ref{lem:ext} and ideas similar to those in the
proof of Theorem~\ref{construction}. But in order to get better
bounds, whenever possible (cases 1, 2 and 3), we start with a cage and
not just with a wcycle. In this way, we can guarantee the girth and one of
the regularities, and then we obtain the desired wcage by using only
one extension operation. In the last case, the girth is guaranteed not
by the initial graph but by the construction itself.

(1) Let $G_{0}$ be an $(a,g)$-cage. Since $a\geq 2$ and $g\geq 3$,
$G_0$ does exist.  In addition, $g(G_0)=g$, this guarantees the girth
of the wgraph that will be constructed.
  
Let $X_{0}=H(G_{0})$, the heavy-subgraph of $G_0$, which is a discrete
graph. Take $d=b$ and then the defect $D = b\cdot|X_{0}|-0$ as in the
Extension Lemma. Now, select $Y_{0}$ to be a $D$-regular graph with
girth $g(Y_{0})\geq \hc{g}$ and minimal order
$|Y_{0}|=\tn(D,\hc{g})$. Use the Extension Lemma to get a $b$-regular
graph $Z_{1}=X_{0}\rtimes_{\varphi}Y_{0}$ as in the
Theorem~\ref{construction}.

Now consider the edges of $Z_{1}$ to be heavy edges and, in each copy
of $X_{0}$, put back the light edges originally present in $G_0$.  Let
us name the resulting wgraph as $G_{1}$. Since $g(G_0)=g$, as in the
proof of Theorem~\ref{construction}, we have that $g(G_1)=g$.
Furthermore, each vertex of $G_1$ is incident with $a$ light and $b$
heavy edges.  Hence, $G_1$ is an $(a,b,g)$-wgraph of order
$n(a,g)\cdot \tn(D,\hc{g})=n(a,g)\cdot \tn(b\cdot
n(a,g),\hc{g})$.

(2) Let $G_0$ be a $(b,\hh{g})$-cage. Since $b\geq 2$ and $g\geq 6$,
$G_{0}$ does exist. The edges of $G_{0}$ will produce the heavy edges
of the constructed wgraph. Let $D = a\cdot |G_0|$ and let $Y_{0}$ be a
$D$-regular graph with girth $g(Y_{0})\geq g$ and minimal order
$|Y_{0}|=\tn(D,g)$.  $Y_{0}$ is the light-subgraph of the sought graph. As
before, we can construct $G_1$, an $(a,b,g)$-wgraph of order
$n(b,\hh{g})\cdot \tn(a\cdot n(b,\hh{g}),g)$.

(3) Construct an $(a,g)$-cage with weight 1 on its edges, take two
copies and complete it with a matching of heavy edges. We claim that
no wcycles with a weight less than 6 are formed, since the new wcycles
contain at least two heavy edges of the matching, and the rest are
light ones, which are also at least two, therefore, the new wcycles have
weight at least 6.

(4) Construct a $b$-regular graph of girth at least 3 and then
consider its edges to be heavy. Take two disjoint copies of that, and complete
it with a matching of light edges, taking care that at least one wcycle
of weight 6 is formed.  As before, no wcycles of weight less than 6 are
formed.
\end{proof}

\section{Moore-like bounds}\label{sec:mbounds}

Much in the way of Moore's lower bounds for cages \cite{HS60}, we may
also provide lower bounds for wcages.  As in the classic case, we
construct a wtree which must be an induced subwgraph of any wcage
of some given parameters, and where all the vertices must be different
to avoid creating wcycles of weight less than $g$. The result is
Theorem~\ref{thm:MBounds} and this section is devoted to prove it.

\textbf{Assume first that $g$ is odd}. Start with a root vertex and
create 
a wtree of depth (wdistance from the root)
$h=\lfloor(g-1)/2\rfloor=(g-1)/2$ whose inner vertices have $a$ and $b$ light incident edges and heavy incident edges, respectively. All of these vertices must be different since, otherwise we would form a wcycle of weight at most
$2h=g-1<g$, which are forbidden. Any $(a,b,g)$-wcage must contain this
wtree as an induced subwgraph, and hence the order of the wtree is a
lower bound for $n(a,b,g)$.

Since we have light and heavy edges, we should create the several
levels of the wtree, considering the wdistance of the vertices from
the root (the root is at level 0), and hence heavy edges skip two
levels at a time as in Figure~\ref{fig:tree229}. We shall consider two
kinds of vertices, the light vertices (in green) and the heavy
vertices (in red): Vertices are light or heavy depending on the weight
of the edge connecting them to their respective parents. The root
vertex is not of any of these kinds, but we shall see that, for
counting purposes, it can be considered light or heavy depending on
the case at hand.

\begin{figure}[h]
\begin{center}
\begin{tikzpicture}[scale = 0.6]
\node[stt] (1) at (0,0) {};
\node[stt,red] (2) at (1,0) {};
\node[stt,red] (3) at (2,0) {};
\node[stt] (4) at (3,0) {};
\node[stt] (5) at (4,0) {};
\node[stt] (6) at (5,0) {};
\node[stt] (7) at (6,0) {};
\node[stt] (8) at (7,0) {};
\node[stt,red] (9) at (8,0) {};
\node[stt,red] (10) at (9,0) {};
\node[stt] (11) at (10,0) {};
\node[stt] (12) at (11,0) {};
\node[stt] (13) at (12,0) {};
\node[stt] (14) at (13,0) {};
\node[stt] (15) at (14,0) {};
\node[stt] (16) at (15,0) {};
\node[stt,red] (17) at (16,0) {};
\node[stt] (18) at (17,0) {};
\node[stt] (19) at (18,0) {};
\node[stt,red] (20) at (19,0) {};
\node[stt] (21) at (0,1.5) {};
\node[stt,red] (22) at (3.5,1.5) {};
\node[stt,red] (23) at (5.5,1.5) {};
\node[stt] (24) at (7,1.5) {};
\node[stt,red] (25) at (10.5,1.5) {};
\node[stt,red] (26) at (12.5,1.5) {};
\node[stt] (27) at (14,1.5) {};
\node[stt] (28) at (15,1.5) {};
\node[stt] (29) at (17,1.5) {};
\node[stt] (30) at (18,1.5) {};
\node[stt] (31) at (1,3) {};
\node[stt] (32) at (8,3) {};
\node[stt,red] (33) at (15,3) {};
\node[stt,red] (34) at (18,3) {};
\node[stt] (35) at (3.5,4.5) {};
\node[stt] (36) at (10.5,4.5) {};
\node[stt,black] (37) at (10.5,6) {};
\draw[thick] (1)--(21)--(31)--(35)--(37)--(36)--(32)--(24)--(8);
\draw[thick] (4)--(22)--(5);
\draw[thick] (6)--(23)--(7);
\draw[thick] (11)--(25)--(12);
\draw[thick] (13)--(26)--(14);
\draw[thick] (15)--(27)--(33)--(28)--(16);
\draw[thick] (18)--(29)--(34)--(30)--(19);
\draw[very thick,red] (2)--(31)--(3);
\draw[very thick,red] (22)--(35)--(23);
\draw[very thick,red] (9)--(32)--(10);
\draw[very thick,red] (25)--(36)--(26);
\draw[very thick,red] (17)--(33)--(37)--(34)--(20);
\end{tikzpicture}
\end{center}
\caption{Moore-like wtree for $a=2, b=2, g=9$.}
\label{fig:tree229}
\end{figure}
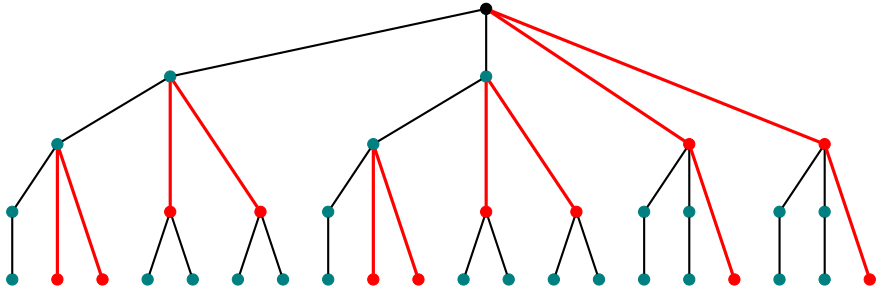

We define $L_{i}$ as the number of light vertices at level $i$ and
$H_{i}$ as the number of heavy vertices at level $i$. Then it should
be clear that the recurrences for light and heavy vertices at level
$i$ are:

\begin{equation}\label{eq:recurrence}
\begin{array}{rcl}
L_{i}&=&(a-1)L_{i-1}+aH_{i-1},\\
H_{i}&=&(b-1)H_{i-2}+bL_{i-2}.
\end{array}
\end{equation}

And that, the base cases are:

\begin{equation}\label{eq:baseoddgirth}
\begin{array}{ll}
L_{0}=1, &H_{0}=0,\\
L_{1}=a,&H_{1}=0.
\end{array}
\end{equation}

Note that in this case the root vertex is considered light ($L_{0}=1$), since it
affects the number of heavy vertices at level 2 according to the
recurrence for $H_{i}$ in (\ref{eq:recurrence}), but it does not affect the light
vertices at level 1 since those are determined by the base cases in (\ref{eq:baseoddgirth}) and
not by the recurrences. 

For future reference, let us name this lower bound.
\begin{equation}\label{eq:m1}
M_{1}:= M_{1}(a,b,g):=\sum_{i=0}^{(g-1)/2} (L_{i}+H_{i}) 
\textup{\hspace{0.5cm}using (\ref{eq:recurrence}) and (\ref{eq:baseoddgirth}), $g$ is odd.}
\end{equation}

\textbf{Now assume $g$ is even}. As before we can construct 
a wtree, and again, its depth must be at most
$h=\lfloor(g-1)/2\rfloor=(g-2)/2$ to guarantee that all of the
vertices are different (assuming there are no wcycles of weight less
than $g$). Although we can not add an additional full level preserving
this guarantee, we can indeed add an additional level but only to the
subwtree of one of the children of the root and still guarantee that
all the vertices are different, this is true since
$h+(h+1) = (g-2)/2+(g-2)/2+1 =g-1<g$. Since we have light and heavy
edges, this can be done in two different ways as shown in
figures~\ref{fig:tree228}(a) and \ref{fig:tree228}(b). There, the
child of the root selected to have an additional level of descendants
is marked with a square box. Any $(a,b,g)$-wcage must contain both of
these wtrees as induced subwgraphs and hence the orders of these wtrees
are both lower bounds for $n(a,b,g)$.

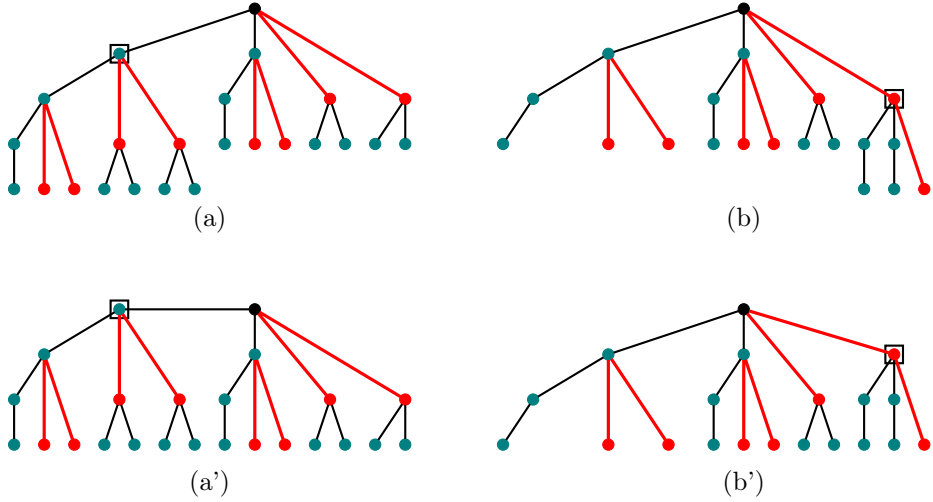
\begin{figure}[h]
\begin{center}
\begin{tikzpicture}
\begin{scope}[scale=0.4] 
\node[] () at (6.5,-1) {(a)};
\node[stt] (1) at (0,0) {};
\node[stt,red] (2) at (1,0) {};
\node[stt,red] (3) at (2,0) {};
\node[stt] (4) at (3,0) {};
\node[stt] (5) at (4,0) {};
\node[stt] (6) at (5,0) {};
\node[stt] (7) at (6,0) {};
\node[stt] (8) at (0,1.5) {};
\node[stt,red] (9) at (3.5,1.5) {};
\node[stt,red] (10) at (5.5,1.5) {};
\node[stt] (11) at (7,1.5) {};
\node[stt,red] (12) at (8,1.5) {};
\node[stt,red] (13) at (9,1.5) {};
\node[stt] (14) at (10,1.5) {};
\node[stt] (15) at (11,1.5) {};
\node[stt] (16) at (12,1.5) {};
\node[stt] (17) at (13,1.5) {};
\node[stt] (18) at (1,3) {};
\node[stt] (19) at (7,3) {};
\node[stt,red] (20) at (10.5,3) {};
\node[stt,red] (21) at (13,3) {};
\node[draw,thick] () at (3.5,4.5) {};
\node[stt] (22) at (3.5,4.5) {};
\node[stt] (23) at (8,4.5) {};
\node[stt,black] (24) at (8,6) {};
\draw[thick] (1)--(8)--(18)--(22)--(24)--(23)--(19)--(11);
\draw[thick] (4)--(9)--(5);
\draw[thick] (6)--(10)--(7);
\draw[thick] (14)--(20)--(15);
\draw[thick] (16)--(21)--(17);
\draw[very thick,red] (2)--(18)--(3);
\draw[very thick,red] (9)--(22)--(10);
\draw[very thick,red] (12)--(23)--(13);
\draw[very thick,red] (20)--(24)--(21);
\end{scope}
\begin{scope}[xshift=6.5cm,scale=0.4] 
\node[] () at (8,-1) {(b)};
\node[stt] (1) at (12,0) {};
\node[stt] (2) at (13,0) {};
\node[stt,red] (3) at (14,0) {};
\node[stt] (8) at (0,1.5) {};
\node[stt,red] (9) at (3.5,1.5) {};
\node[stt,red] (10) at (5.5,1.5) {};
\node[stt] (11) at (7,1.5) {};
\node[stt,red] (12) at (8,1.5) {};
\node[stt,red] (13) at (9,1.5) {};
\node[stt] (14) at (10,1.5) {};
\node[stt] (15) at (11,1.5) {};
\node[stt] (16) at (12,1.5) {};
\node[stt] (17) at (13,1.5) {};
\node[stt] (18) at (1,3) {};
\node[stt] (19) at (7,3) {};
\node[stt,red] (20) at (10.5,3) {};
\node[stt,red] (21) at (13,3) {};
\node[draw,thick] () at (13,3) {};
\node[stt] (22) at (3.5,4.5) {};
\node[stt] (23) at (8,4.5) {};
\node[stt,black] (24) at (8,6) {};
\draw[thick] (8)--(18)--(22)--(24)--(23)--(19)--(11);
\draw[thick] (14)--(20)--(15);
\draw[thick] (1)--(16)--(21)--(17)--(2);
\draw[very thick,red] (9)--(22)--(10);
\draw[very thick,red] (12)--(23)--(13);
\draw[very thick,red] (20)--(24)--(21)--(3);
\end{scope}
\begin{scope}[yshift=-4cm,scale=0.4] 
\node[] () at (6.5,0.2) {(a')};
\node[stt] (1) at (0,1.5) {};
\node[stt,red] (2) at (1,1.5) {};
\node[stt,red] (3) at (2,1.5) {};
\node[stt] (4) at (3,1.5) {};
\node[stt] (5) at (4,1.5) {};
\node[stt] (6) at (5,1.5) {};
\node[stt] (7) at (6,1.5) {};
\node[stt] (8) at (0,3) {};
\node[stt,red] (9) at (3.5,3) {};
\node[stt,red] (10) at (5.5,3) {};
\node[stt] (11) at (7,1.5) {};
\node[stt,red] (12) at (8,1.5) {};
\node[stt,red] (13) at (9,1.5) {};
\node[stt] (14) at (10,1.5) {};
\node[stt] (15) at (11,1.5) {};
\node[stt] (16) at (12,1.5) {};
\node[stt] (17) at (13,1.5) {};
\node[stt] (18) at (1,4.5) {};
\node[stt] (19) at (7,3) {};
\node[stt,red] (20) at (10.5,3) {};
\node[stt,red] (21) at (13,3) {};
\node[draw,thick] () at (3.5,6) {};
\node[stt] (22) at (3.5,6) {};
\node[stt] (23) at (8,4.5) {};
\node[stt,black] (24) at (8,6) {};
\draw[thick] (1)--(8)--(18)--(22)--(24)--(23)--(19)--(11);
\draw[thick] (4)--(9)--(5);
\draw[thick] (6)--(10)--(7);
\draw[thick] (14)--(20)--(15);
\draw[thick] (16)--(21)--(17);
\draw[very thick,red] (2)--(18)--(3);
\draw[very thick,red] (9)--(22)--(10);
\draw[very thick,red] (12)--(23)--(13);
\draw[very thick,red] (20)--(24)--(21);
\end{scope}
\begin{scope}[xshift=6.5cm,yshift=-4cm,scale=0.4] 
\node[] () at (8,0.2) {(b')};
\node[stt] (1) at (12,1.5) {};
\node[stt] (2) at (13,1.5) {};
\node[stt,red] (3) at (14,1.5) {};
\node[stt] (8) at (0,1.5) {};
\node[stt,red] (9) at (3.5,1.5) {};
\node[stt,red] (10) at (5.5,1.5) {};
\node[stt] (11) at (7,1.5) {};
\node[stt,red] (12) at (8,1.5) {};
\node[stt,red] (13) at (9,1.5) {};
\node[stt] (14) at (10,1.5) {};
\node[stt] (15) at (11,1.5) {};
\node[stt] (16) at (12,3) {};
\node[stt] (17) at (13,3) {};
\node[stt] (18) at (1,3) {};
\node[stt] (19) at (7,3) {};
\node[stt,red] (20) at (10.5,3) {};
\node[stt,red] (21) at (13,4.5) {};
\node[draw,thick] () at (13,4.5) {};
\node[stt] (22) at (3.5,4.5) {};
\node[stt] (23) at (8,4.5) {};
\node[stt,black] (24) at (8,6) {};
\draw[thick] (8)--(18)--(22)--(24)--(23)--(19)--(11);
\draw[thick] (14)--(20)--(15);
\draw[thick] (1)--(16)--(21)--(17)--(2);
\draw[very thick,red] (9)--(22)--(10);
\draw[very thick,red] (12)--(23)--(13);
\draw[very thick,red] (20)--(24)--(21)--(3);
\end{scope}%
\end{tikzpicture}%
\end{center}%
\caption{Moore-like wtrees for $a=2, b=2, g=8$.}%
\label{fig:tree228}%
\end{figure}

In order to count the vertices of these wtrees, we may proceed as
before, but the additional partial levels would require us to resort
to two additional sets of recurrences to compute how many light and
heavy edges are present in the last two levels of the subwtrees that
were expanded, so we can finally count the number of vertices in the
additional partial levels. A better idea is to move the selected
childs one level up, as in figures~\ref{fig:tree228}(a') and
\ref{fig:tree228}(b'). In this way, all the leaves are aligned and the
recurrences in (\ref{eq:recurrence}) hold good for all levels
($i\geq 2$) and all cases. Then we only have to check which the new
base cases are.  It should be clear that the new base cases when the
selected child is light (as in Figure~\ref{fig:tree228}(a')), are:

\begin{equation}\label{eq:baseevengirth1}
\begin{array}{ll}
L_{0}=2, &H_{0}=0,\\
L_{1}=2(a-1),&H_{1}=0.
\end{array}
\end{equation}

Also, the base cases when the selected child is heavy (as in
Figure~\ref{fig:tree228}(b')) are:

\begin{equation}\label{eq:baseevengirth2}
\begin{array}{ll}
L_{0}=0, &H_{0}=1,\\
L_{1}=a,&H_{1}=1.
\end{array}
\end{equation}

Note that the root vertex is considered light ($L_{0}=2$) in
(\ref{eq:baseevengirth1}) and heavy ($H_{0}=1$) in
(\ref{eq:baseevengirth2}) this affects the number of heavy vertices at
level 2 as computed with the recurrences (\ref{eq:recurrence}). Also,
the selected child in both cases moves only one level up, so the
selected child is at level 0 in Figure~\ref{fig:tree228}(a') and it is
at level 1 in Figure~\ref{fig:tree228}(b') in agreement with equations
(\ref{eq:baseevengirth1}) and (\ref{eq:baseevengirth2}).  Let us name
these two new lower bounds:

\begin{equation}\label{eq:m2}
M_{2}:= M_{2}(a,b,g):=\sum_{i=0}^{(g-2)/2} (L_{i}+H_{i}) 
\textup{\hspace{0.5cm}using (\ref{eq:recurrence}) and (\ref{eq:baseevengirth1}), $g$ is even.}
\end{equation}

\begin{equation}\label{eq:m3}
M_{3}:= M_{3}(a,b,g):=\sum_{i=0}^{(g-2)/2} (L_{i}+H_{i}) 
\textup{\hspace{0.5cm}using (\ref{eq:recurrence}) and (\ref{eq:baseevengirth2}), $g$ is even.}
\end{equation}

Let $n=n(a,b,g)$. Note that parity forbids both $an\equiv 1$ and
$bn\equiv 1$, hence we can add one to each lower bound whenever the
bound itself is odd and either $a$ or $b$ is odd. Hence, for
$i\in\{1,2,3\}$, we define:
$$M_{i}^{+}= \begin{cases} 
M_{i}+1 &\textup{ if $M_{i}$ is odd and either $a$ or $b$ is odd,}\\
M_{i} &\textup{ otherwise.}
\end{cases}$$

Therefore, in this section we have proven that
$n(a,b,g)\geq M_{1}^{+}$, for odd $g$, and that
$n(a,b,g)\geq \max\{M_{2 }^{+},M_{3}^{+}\}$, for even $g$. However, we
have found in practice that almost always we have that
$M_{2}\geq M_{3}$ (and hence that $M_{2}^{+}\geq M_{3}^{+}$). The
only relevant exception we have found is
$M_{3}(1,2,10)=15 > 14 =M_{2}(1,2,10)$ (other exceptions occur when
the $(a,b,g)$-wcage does not even exist). Hence we prefer to state the
theorem that we have proven in this section as follows:

\begin{theorem}\label{thm:MBounds} Let $a\geq 1$, $b\geq 1$,
  $g\geq 3$. Then we have:
$$\begin{array}{llll}
n(a,b,g) &\geq &M_{1}^{+} &\textup{ when $g$ is odd,}\\
n(a,b,g) &\geq &M_{2}^{+} &\textup{ when $g$ is even.}\\
n(a,b,g) &\geq &M_{3}^{+}=16 &\textup{ when $a=1,b=2,g=10$.}\\
\end{array}$$
\end{theorem}

We shall collectively denote these lower bounds ($M_{1}^{+}$, $M_{2}^{+}$ and $M_{3}^{+}$, according to the cases as in the previous Theorem) simply by $n_0(a,b,g)$. Hence the previous Theorem says that $n(a,b,g)\geq n_{0}(a,b,g)$. Note that the standard Moore lower bound for ordinary cages is $n_0(r,g) = n_0(r,0,g)$, and that the standard Moore trees are the same as the trees in this section in the case $b=0$. Whenever we have an $(a,b,g)$-wgraph of order $n$, we shall say that its \emph{excess} is $n-n_0(a,b,g)$. It is straightforward to verify that these lower bounds give the following closed formulas (we used GAP \cite{GAP4} for the required symbolic computations):
$$\begin{array}{ll}
g=3: & M_{1} = a+1\\ 
g=5: & M_{1} = a^2+b+1\\ 
g=7: & M_{1} = a^3-a^2+2ab+a+b+1\\ 
g=9: & M_{1} = a^4-2a^3+3a^2b+2a^2+b^2+1\\ 
g=11: & M_{1} = a^5-3a^4+4a^3b+4a^3-3a^2b+3ab^2-2a^2+b^2+a+1\\ 
&\\
g=4: & M_{2} = 2a\label{eq:otherbound}\\ 
g=6: & M_{2} = 2a^2-2a+2b+2\\ 
g=8: & M_{2} = 2a^3-4a^2+4ab+4a\\ 
g=10: & M_{2} = 2a^4-6a^3+6a^2b+8a^2-4ab+2b^2-4a+2\\
g=12: & M_{2} = 2a^5-8a^4+8a^3b+14a^3-12a^2b+6ab^2-12a^2+4ab+6a
\end{array}$$

These lower bounds are not great for $g=3$ or $g=4$ as we also have
the lower bound $n(a,b,g)\geq a+b+1$ from
Lemma~\ref{lem:trivialbound}, which often surpasses both of these
bounds. In the following two sections we shall determine $n(a,b,g)$
for $g=3,4$. After that (sections \ref{sec:g56} and \ref{sec:exp}), we
shall see that these lower bounds are much better for $g=5,6$, and
that they are relevant for $g\geq 7$.

\section{Weighted cages of girth 3}\label{sec:g3}

We shall prove Theorem~\ref{thm:mainthree} which characterizes
$n(a,b,3)$.  Recall by Lemma~\ref{lem:trivialbound} that, in general,
we have that $n(a,b,g)\geq a+b+1$ and that, when $ab\equiv 1$, we have
$n(a,b,g)\geq a+b+2$. Hence, Theorem~\ref{thm:mainthree} says that
these lower bounds are sharp except for the first two conditions in
the Theorem.  Note that a wcycle of girth $g=3$ must use only light
edges and hence the heavy edges can be placed freely in our wgraph
never affecting the already minimal girth of the wgraph.

A frequently used idea is that if $n=a+b+1$ and $L$ is already
$a$-regular of girth $g=3$, then its complement $H=\overline{L}$ can
always be used to obtain the desired wgraph
$G=(L,H)=(L,\overline{L})$.

\begin{theorem}\label{thm:mainthree}
For each $a\geq 0$ and $b\geq 0$ we have that 
$$n(a,b,3)=\begin{cases}
\infty& \tr{if } a<2 \\ 
6     & \tr{if } a=2 \tr{  and } b\in \{1,2\} \\
a+b+1 & \tr{if }  a=2 \tr{  and } b\not\in \{1,2\}\\
a+b+1 & \tr{if } a\geq 3 \tr{  and } ab\equiv 0 \\
a+b+2 & \tr{if } a\geq 3 \tr{  and } ab\equiv 1
\end{cases}$$
\end{theorem}
\begin{proof}
\textbf{Case 1  [$a<2$]:} Immediate form Lemma~\ref{existence}.

\textbf{Case 2 [$a=2$ and $b\in \{1,2\}$]:} Let $G$ be an
$(a,b,3)$-wcage and let $L$ its light-subgraph.  Since $a=2$, $L$ is a
disjoint union of cycles. Since $g=3$ one of these cycles must be a
triangle. Since $b>0$, we need at least two cycles in $L$ and since
cycles have length at least 3, it follows that $n(2,b,3)\geq 6$ in
this case. It should be clear that the required heavy edges can always
be added to the disjoint union of two triangles. Hence $n(2,b,3)=6$
for $b\in \{1,2\}$.

\textbf{Case 3 [$a=2$ and $b\not\in \{1,2\}$]:} Take $n=a+b+1
=b+3$. For $b=0$ a triangle $G$ will work. For $b\geq 3$, we can take
$L$ as the disjoint union of a triangle and a cycle of length
$b$. Then $G=(L,\overline{L})$ is the required wgraph.

\textbf{Case 4 [$a\geq 3$ and $ab\equiv 0$]:} Take $n=a+b+1$.  

Assume first that $a\equiv b\equiv 0$. Then $n\equiv 1$, $a\geq 4$ and
$n\geq 5$. Take $F_{i}$ as in Lemma~\ref{lem:2factor} and take
$L=\bigcup_{i=1}^{\hh{a}}F_{i}$. Then $G=(L,\overline{L})$ is the
required wgraph.

Assume now that $a\not\equiv b$. Then $n\equiv 0$ and $n\geq 4$. Take
$\tilde{F}_{i}$ as in Lemma~\ref{lem:1factor} and
$L=\bigcup_{0}^{a-1}\tilde{F}_{i}$. Then $G=(L,\overline{L})$ is the
required wgraph.

\textbf{Case 5 [$a\geq 3$ and $ab\equiv 1$]:} In this case, we have
$n\geq a+b+2$ by Lemma~\ref{lem:trivialbound}, but we can indeed
provide a wgraph on $a+b+2$ vertices with the required parameters:
Take $n=a+b+2\equiv 0$ and take $\tilde{F}_{i}$ as in
Lemma~\ref{lem:1factor}. Now take $L=\bigcup_{i=0}^{a-1}\tilde{F}_{i}$
and $H=\bigcup_{i=a}^{n-3}\tilde{F}_{i}$. Since $n-3=a+b-1$, $H$ is
$b$-regular and $G=(L,H)$ is the required wgraph.
\end{proof}

\section{Weighted cages of girth 4}\label{sec:g4}

In this section we prove Theorem~\ref{thm:mainfour} that determine the
values $n(a,b,4)$ for each $a\geq 0$ and $b\geq 0$. Besides the lower
bounds in Lemma~\ref{lem:trivialbound}, we also have the bound
$n(a,b,4)\geq M_{2}=2a$ from page~\pageref{eq:otherbound}. Hence,
Theorem~\ref{thm:mainfour} says that $n(a,b,4)$ stays close to these
bounds except when $a<2$. This time we have to avoid triangles in
$L=L(G)$ and also, we have to guarantee a wcycle of weight 4 in $G$,
which may be formed by four light edges or by two light edges and a
heavy edge. Once this is achieved, we can add heavy edges freely to
$G$ without changing the girth of $G$. Also, if $L$ is already
$a$-regular of girth $4$ and order $n=a+b+1$, then we can always get
the required wgraph by taking $G=(L,\overline{L})$.

\begin{lemma}\label{lem:leqab2} 
  If $3\leq a \leq b$ and $a\equiv b$,  then $n(a,b,4)\leq a+b+2$.
\end{lemma}
\begin{proof} 

  Let $n=a+b+2\equiv 0$ and $m=\hh{n}\geq a+1$. Note that $m\leq
  b+1$. Let $X,Y$ be the parts of the complete bipartite graph
  $K_{m,m}$ considered in Lemma~\ref{lem:1factorBip} and also let
  $\hat{F}_i$ as in that lemma. Take
  $L=\bigcup_{i=0}^{a-1}\hat{F}_{i}$. Clearly $L$ is $a$-regular, of
  girth 4 and order $n$. For $H$, take all possible edges within $X$
  and all possible edges within $Y$. At this point, $H$ is already
  $(m-1)$-regular, since $m\leq b+1$, $H$ could already be
  $b$-regular, but if not, the extra edges may obtained by adding to
  $H$ the edges of $\bigcup_{i=a}^{m-2}\hat{F}_{i}$. Since
  $(m-1)+(m-2-a+1) = 2m-a-2 = b$, $H$ is now $b$-regular and $G=(L,H)$
  is the required wgraph.
\end{proof}

\begin{lemma}\label{lem:newlemma} 
If $3 \leq a \leq b, a\equiv 0$ and $b> \hh{3a}-2$ then $n(a,b,g)=a+b+1$.
\end{lemma}
\begin{proof}
  As before, it will suffice to construct an $a$-regular graph $L$ of
  girth 4 and order $n=a+b+1$. By our hypotheses, we have
  $b\geq \hh{3a}-1$ and hence $n\geq \hh{5a}$.  Let $r$ and $k$ be
  integers such that $n = \hh{5a}+\hh{a}r+k$ with $r\geq 0$ and
  $0\leq k < \hh{a}$.

  Assume first that $r=0$. Figure~\ref{fig:ab4first} shows a diagram
  of our construction. There, each node represents an independent set
  of vertices of the indicated cardinality ($\hh{a}$ for the round
  nodes and $\hh{a}+k$ or $\hh{a}-k$ for the others). A solid line in
  the diagram, means to add all possible edges between the
  corresponding independent sets. The dashed line, means to add edges
  between the corresponding independent sets in such a way as to form
  an $\hh{a}$-regular bipartite graph among them. It is straightforward to
verify that the just constructed graph $L$ is $a$-regular, of girth 4
and of order $n=\hh{5a}+k$.

\begin{figure}[h]
\begin{center}
\begin{tikzpicture}[scale = 1]
\node[draw,thick] (1) at (2,2) {$\hh{a}+k$};
\node[draw,thick] (2) at (6,2) {$\hh{a}+k$};
\node[draw,thick,circle] (3) at (0,1) {$\hh{a}$};
\node[draw,thick,circle] (4) at (8,1) {$\hh{a}$};
\node[draw,thick] (5) at (4,0) {$\hh{a}-k$};
\draw[thick,dashed] (1)-- node[above]{$\hh{a}$} (2);
\draw[thick] (1)--(3)--(5)--(4)--(2);
\end{tikzpicture}
\end{center}
\caption{Construction of (a,b,4)-wcages for $n=\hh{5a}+k$.}
\label{fig:ab4first}
\end{figure}
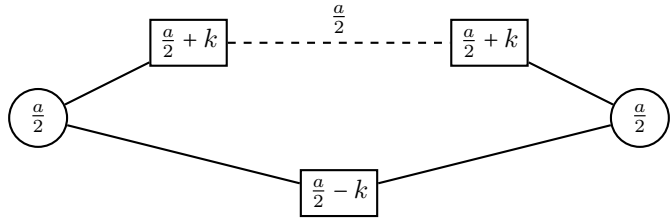

Assume now, that $r\geq 1$. Figure~\ref{fig:ab4second} shows a diagram
of our construction. As before, our $n= \hh{5a}+\hh{a}r+k$ vertices
are partitioned into independent sets indicated by the nodes in the
diagram: round nodes contain $\hh{a}$ vertices and the other node
contains $k$ vertices, the number of gray nodes must be $r$ (and hence
there is at least one) always forming a path as indicated in the
diagram (hence, Figure~\ref{fig:ab4second} illustrates the case
$r=4$).  Again, the solid lines means to add all possible edges there
and the dashed lines means to add edges in such a way as to form an
$s$-regular bipartite graph there (the value of $s$ is indicated by
the number near the dashed line).  It is then straightforward to
verify that the just constructed graph $L$ is $a$-regular, of girth 4
and of order $n=\hh{5a}+\hh{a}r+k$.

\begin{figure}[h]
\begin{center}
\begin{tikzpicture}[scale = 1]
\node[draw,thick] (1) at (0,2) {$k$};
\node[draw,thick,circle] (2) at (2,4) {$\hh{a}$};
\node[draw,thick,circle] (3) at (2,0) {$\hh{a}$};
\node[draw,thick,circle] (4) at (6,4) {$\hh{a}$};
\node[draw,thick,circle] (5) at (6,0) {$\hh{a}$};
\node[draw,thick,circle] (6) at (4,2) {$\hh{a}$};
\node[draw,thick,circle,fill=gray!60] (10) at (10,4) {$\hh{a}$};
\draw[thick] (10) to[bend left] (5); 
\node[draw,thick,circle,fill=gray!60] (7) at (7.5,0.5) {$\hh{a}$};
\node[draw,thick,circle,fill=gray!60] (8) at (8.67,1.33) {$\hh{a}$};
\node[draw,thick,circle,fill=gray!60] (9) at (9.55,2.55) {$\hh{a}$};
\draw[thick,dashed] (2)-- node[below]{$\hh{a}-k$} (4)
  -- node[right]{$k$}(5)
  --node[above]{$\hh{a}-k$}(3);
\draw[thick] (2)--(1)--(3)--(6)--(4);
\draw[thick] (10) to[bend right] (2);
\end{tikzpicture}
\end{center}
\caption{Construction of (a,b,4)-wcages for
  $n=\hh{5a}+\hh{a}r+k$ with $r=4$.}
\label{fig:ab4second}
\end{figure}
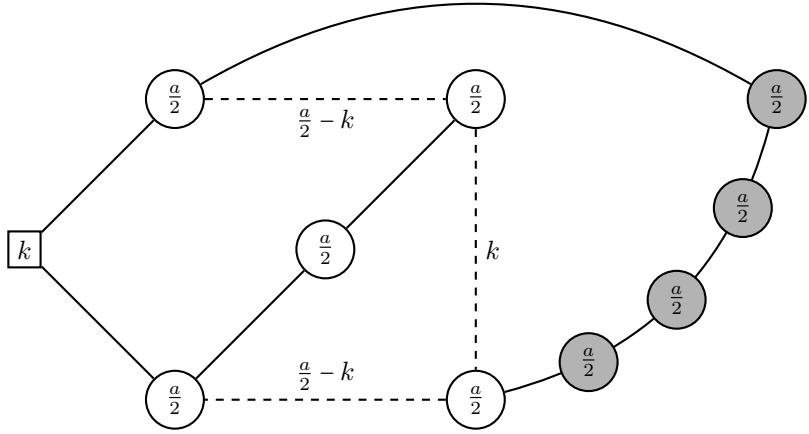

\end{proof}

\begin{lemma}\label{lem:hasatriangle} 
  Let $a\equiv 0$ and $n\equiv 1$ with $2a<n<\hh{5a}$. Then every
  $a$-regular graph $L$ on $n$ vertices has a triangle.
\end{lemma}

\begin{proof} 
  Let $L$ be a triangle-free $a$-regular graph with $n$ vertices, $n$
  an odd integer and $2a< n<\hh{5a}$.

  Let $x$ and $y$ be two adjacent vertices and let
  $A_x=N(x)\setminus \{y\}$ and $A_y=N(y)\setminus \{x\}$. Then,
  $A_x\cap A_y$ is empty. Hence,
  $I:=V\setminus (A_x\cup A_y\cup \{x,y\})$ has $n-2a\leq \hh{a}-1$
  vertices, since $|A_x|=|A_y|=a-1$. Moreover, $|I|$ is odd and each
  vertex $u\in I$ has at least $\hh{a}+2$ neighbors not in $I$.

  We first prove that for each vertex $u$ in $I$, $N(u)\cap A_x$ is
  empty or $N(u)\cap A_y$ is empty.  For the sake of a contradiction,
  let us assume without loss of generality that a vertex $u\in I$ has
  $k$ neighbors in $I$, $l_x$ neighbors in $A_x$ and $l_y$ neighbors
  in $A_y$, with $l_x\geq l_y\geq 1$. Then, $a=k+l_x+l_y$.  Let
  $w\in N(u)\cap A_y$. Then
$$N(w)\subseteq \{y\}\cup (A_x\setminus N(u))\cup (I\setminus N(u)).$$
Thus, as $|A_x|=a-1$, we get that 
$$a\leq 1+|A_x|-|A_x\cap N(u)|+|I|-|I\cap N(u)|=a-l_x+|I|-k.$$
Hence, $l_x+k\leq |I|$. But, since $l_x\geq l_y$ and $k+l_x+l_y=a$, we
get the contradiction $2|I|\geq a$.

This property allows us to split $I$ into the sets $I_x$ and $I_y$,
where $I_x$ (resp. $I_y$) contains all vertices in $I$ having a
neighbor in $A_x$ (resp. $A_y$).

Let $u\in I_x$. Then, $u$ has no neighbors in $A_y$ and it has at most
$\hh{a}-2$ neighbors in $I$. Hence, it has at least $\hh{a}+2$
neighbors in $A_x$. Therefore, the set $I_x$ is independent.  A
similar argument shows that the set $I_y$ is independent as well.

For sets $X$ and $Y$, let $e(X,Y)$ denotes the number of edges between
$X$ and $Y$. We have that $$a|I_x|=e(I_x,A_x)+e(I_x,I_y),$$
$$a|I_y|=e(I_y,A_y)+e(I_y,I_x),$$
and
$$(a-1)^2=e(A_x,A_y)+e(A_x,I_x)=e(A_x,A_y)+e(A_y,I_y).$$
These equalities imply that $|I_x|=|I_y|$ which is not possible when
$I$ is odd.
\end{proof}

\begin{theorem}\label{thm:mainfour}
For each $a\geq 0$ and $b\geq 0$ we have that 
$$n(a,b,4)=\begin{cases}
\infty& \tr{ if } a<2 \\ 
2a     & \tr{ if } a=2 \tr{  and } b=0 \\
a+b+1   & \tr{ if } a=2 \tr{  and } b\geq 1 \\
2a    & \tr{ if } 3 \leq a > b \tr{ and } ab\equiv 0 \\  
2a+2  & \tr{ if } 3 \leq a > b \tr{ and }   ab\equiv 1\\
a+b+1 & \tr{ if } 3 \leq a \leq b \tr{ and } a\not\equiv b\\
a+b+2 & \tr{ if } 3 \leq a \leq b , a\equiv b\equiv 1\\
a+b+2 & \tr{ if } 3 \leq a \leq b , a\equiv b\equiv 0\tr{ and } b\leq \hh{3a}-2\\
a+b+1 & \tr{ if } 3 \leq a \leq b , a\equiv b\equiv 0\tr{ and } b> \hh{3a}-2.
\end{cases}$$
\end{theorem}
\begin{proof} Recall that $n(a,b,4)\geq a+b+1$ and that
  $n(a,b,4)\geq 2a$, so constructions of these orders immediately
  determine $n(a,b,4)$.

\textbf{Case 1 [$a<2$]:} Immediate from Lemma~\ref{existence}. 

\textbf{Case 2 [$a=2$, $b=0$]:} Immediate.

\textbf{Case 3 [$a=2$, $b\geq 1$]:} Take $n=a+b+1\geq 4$ and $L$ as a
cycle of length $n$. Take $G=(L,\overline{L})$. The girth $4$ in $G$
is guaranteed by any two consecutive edges in the cycle and the
corresponding heavy chord. Thus $G$ is the required wgraph.

\textbf{Case 4 [$3 \leq a > b$ and $ab\equiv 0$]:} Take $n=2a$ and
$L=K_{a,a}$. Since $a\geq 3$, $L$ already has girth $4$. Let $X$ and
$Y$ be the independent parts of $L$ on $a$ vertices each. Since $a>b$,
we can always put a $b$-regular graph in each of $X$ and $Y$: it is
immediate for $a=3$; use Lemma~\ref{lem:1factor} when
$4\leq a\equiv 0$ or use Lemma~\ref{lem:2factor} when
$5\leq a\equiv 1$, $b\equiv 0$. Let $H$ be the disjoint union of these
two $b$-regular graphs, then $G=(L,H)$ is the sought wgraph.

\textbf{Case 5 [$3 \leq a > b$ and $ab\equiv 1$]:} If there is a
wgraph $G$ with the required parameters and $|G|=2a$, by Turán's
Theorem, we must have $L =L(G)= K_{a,a}$. But then, since
$a\equiv b \equiv 1$, parity forbids to add the required heavy edges
to the parts $X$, $Y$ of $L$ to obtain $G$. Also, since $a\equiv 1$,
parity forbids $|G|=|L|=2a+1$. Hence $n(a,b,4)\geq 2a+2$ in this
case. Let $n=2a+2$, let $M$ be a matching of $K_{a+1,a+1}$ and take
$L=K_{a+1,a+1}-M$. Then $4\leq a+1 \equiv 0$ and by
Lemma~\ref{lem:1factor} we can add the required heavy edges to the
parts $X$, $Y$ of $L$ to obtain a $b$-regular $H$. Hence $G=(L,H)$ is
the required wgraph.

\textbf{Case 6 [$3 \leq a \leq b$ and $a \not\equiv b$]:} Take
$n=a+b+1\equiv 0$ and $m=\hh{n} > a$.  Take $\hat{F_{i}}$ as in
Lemma~\ref{lem:1factorBip}. Define $L=\bigcup_{i=0}^{a-1}\hat{F}_i$, then $L$
is $a$-regular of girth $4$. Hence $G=(L,\overline{L})$ is the
required wgraph.

\textbf{Case 7 [$3 \leq a \leq b$ and $a\equiv b\equiv 1$]:} Parity
forbids $|G|=a+b+1$. Hence $n(a,b,4)=a+b+2$ by Lemma~\ref{lem:leqab2}.

\textbf{Case 8 [$3 \leq a \leq b$, $a\equiv b\equiv 0$ and
  $b\leq \hh{3a}-2$]:} Assume first that $n=a+b+1$ and that $G=(L,H)$
is an $(a,b,4)$-wcage on $n$ vertices. Note that $n\equiv 1$. By our
hypotheses, we have that $n=a+b+1 \leq a+\hh{3a}-2+1 < \hh{5a}$ and
that $n=a+b+1 \geq 2a+1 > 2a$. Hence, by Lemma~\ref{lem:hasatriangle},
$L$ has a triangle, which is a contradiction. It follows that $n(a,b,4)\geq a+b+2$ and by Lemma~\ref{lem:leqab2}, that $n(a,b,4)=a+b+2$.

\textbf{Case 9 [$3 \leq a \leq b$, $a\equiv b\equiv 0$ and
  $b > \hh{3a}-2$]:} Immediate from Lemma~\ref{lem:newlemma}.
\end{proof}

We find interesting the following reinterpretation of the results in this section: 
 
\begin{theorem}\label{t:aregularexists} For each $a\geq 3$ there is an
  $a$-regular graph with girth four and $n$ vertices if and only
  if any of the following cases holds.
\begin{enumerate}
\item $n \equiv 0$ and $n\geq 2a$ or,
\item $n \equiv 1$ and  $a\equiv 0$ and $n\geq \hh{5a}$.
\end{enumerate} 
\end{theorem}
\begin{proof} Let $a$, $n$ as in the statement. Assume $L$ is an
  $a$-regular graph of girth 4 and order $n$ (if it exists). Since
  $n(a,b,4)\geq 2a$, no such $L$ exists for $|L|<2a$.

  Assume first that $n\equiv 0$ and $n\geq 2a$, then take $m=\hh{n}$
  and $\hat{F}_{i}$ as in Lemma~\ref{lem:1factorBip}, now
  $L=\bigcup_{i=0}^{a-1}\hat{F}_i$ is the required graph. Note that
  parity forbids $n\equiv a\equiv 1$.  Assume next that $n\equiv 1$,
  $a\equiv 0$ and $2a < n < \hh{5a}$, then, by
  Lemma~\ref{lem:hasatriangle}, $L$ does not exist.  Finally, if
  $n\equiv 1$, $a\equiv 0$ and $n\geq \hh{5a}$, take $b=n-a-1$. Then
  $b\geq \hh{3a}-1> \hh{3a}-2$. By Lemma~\ref{lem:newlemma}, there is
  an $(a,b,g)$-wgraph on $a+b+1$ vertices. Then $L=L(G)$ is the
  required graph.
\end{proof}

\section{Weighted cages of girth 5 and 6}\label{sec:g56}

Contrary to the cases $g=3$ and $g=4$, our Moore-like bounds in
Theorem~\ref{thm:MBounds} are very good for $g=5$ and $g=6$. Indeed we
shall see in the next section that for these values of $g$, $n(a,b,g)$
coincides with the corresponding Moore-like bound for all the finite
values that we could compute, except for
$n(4,1,5)=20>18 =M_{1}^{+}(4,1,5)$.  The following theorem proves that
this is indeed the case at least for $a=1,2$:

\begin{theorem}\label{a12g560}
  If $n(a,b,5)<\infty$ and $a\in \{1,2\}$, then
  $n(a,b,5)=M_{1}^{+}(a,b,5)$. Also, if $n(a,b,6)<\infty$ and
  $a\in \{1,2\}$ then  $n(a,b,6)=M_{2}^{+}(a,b,6)$.
\end{theorem}

For the reader's convenience and using the polynomials in page
\pageref{eq:otherbound}, we restate the previous theorem in the
following equivalent form:

\begin{theorem}\label{a12g56}
The following relations hold: 

\begin{tabular}{ll}
$n(1,b,5)=b+2$ & \textup{for $2\leq b \equiv 0$},\\
$n(1,b,5)=b+3$ & \textup{for $3\leq b \equiv 1$},\\
$n(2,b,5)=b+5$ & \\
$n(1,b,6)= 2b+2$ & \textup{for $b\geq 1$},\\
$n(2,b,6)= 2b+6$, & \\
\end{tabular}
\end{theorem}
\begin{proof}
  Since the values match the lower bounds $M_{1}^{+}(a,b,5)$ or
  $M_{2}^{+}(a,b,6)$, it will suffice to give constructions matching
  these values.

  \textbf{Case 1 [$a=1$, $g=5$ and $2 \leq b\equiv 0$]:} Take
  $n=b+2\equiv 0$ and take $L=\tilde{F}_{0}$ (see
  Lemma~\ref{lem:1factor}). Then $G=(L,\overline{L})$ guarantees
  $n(a,b,g)\leq b+2$.

  \textbf{Case 2 [$a=1$, $g=5$ and $3 \leq b\equiv 1$]:} Take
  $n=b+3\equiv 0$, $L=\tilde{F}_{0}$ and
  $H=\bigcup_{i=1}^{n-3}\tilde{F}_{i}$. Then $G=(L,H)$ guaranties
  $n(a,b,g)\leq b+3$.

  \textbf{Case 3 [$a=2$ and $g=5$]:} Take $n=b+5$. Note
  that $n$ may be even or odd. Take $L=C_{n}$, the $n$-cycle. Let
  $H=\overline{L^2}$ (in this case, $L^{2}\cong C_{n}(1,2)$ is the
  circulant on $n$ vertices with jumps 1 and 2). Now $G=(L,H)$
  guaranties $n(a,b,g)\leq b+5$. 


  \textbf{Case 4 [$a=1$, $g=6$ and $b\geq 1$]:} Take
  $n=2b+2\equiv 0$ and $m=\hh{n}=b+1$. Take $H=K_{m}\cupdot K_{m}$
  and let $L$ be a matching between these two complete subgraphs. Then
  $G=(L,H)$ guaranties $n(a,b,g)\leq 2b+2$. 

  \textbf{Case 5 [$a=2$ and $g=6$]:} Take
  $n=2b+6\equiv 0$ and $m=\hh{n}=b+3$. Take
  $H=\overline{C_m}\cupdot \overline{C_m}$ and let $L$ be a $2m$-cycle
  zigzagging between these two complements of cycles, taking care that
  no two consecutive edges of $L$ join two adjacent vertices in
  $H$. Then $G=(L,H)$ guaranties $n(a,b,g)\leq 2b+6$. 
%
%
%
\end{proof}

\section{Experimental results}\label{sec:exp}

We used computerized, exhaustive searches using backtracking with
symmetry reductions to obtain the values of $n(a,b,g)$ in the
following tables. The experimental results for the cases $g=3$ and
$g=4$ coincide with the characterizations in the respective sections,
and hence they are omitted here. We also omit the cases $a=0$ and
$b=0$ since those were already characterized in the preliminaries
section. Blank squares are unknown values. In all these cases the
computed values differ by either 0, 2 or 4 from the respective
Moore-like bounds in Theorem~\ref{thm:MBounds}. When the difference is 2
the number in the table is in boldface and black, when the difference
is 4 the number in the table is in blue.  We used GAP~\cite{GAP4} and
YAGS~\cite{YAGS} for these computations.\bigskip

\noindent\begin{minipage}[t]{7cm}
\noindent{}$n(a,b,5)$\\
\begin{tabular}{|r||r|r|r|r|r|r|r|r|}
\hline 
$a\backslash b$&1 &2 &3 &4 &5 &6 &7 &8 \\ \hline\hline
1 &$\infty$ &4 &6 &6 &8 &8 &10 &10 \\ \hline
2 &6 &7 &8 &9 &10 &11 &12 &13 \\ \hline
3 &12 &12 &14 &14 & 16 & 16 &&   \\ \hline
4 &\textbf{20} &19 &20 & 21 &&&&    \\ \hline
\end{tabular}\bigskip

\noindent{}$n(a,b,6)$\\
\begin{tabular}{|r||r|r|r|r|r|r|r|r|}
\hline
$a\backslash b$ &1 &2 &3 &4 &5 &6 &7 &8 \\ \hline\hline
1 &4 &6 &8 &10 &12 &14 &16 &18 \\ \hline
2 &8 &10 &12 &14 &16 &18 &20 &  \\ \hline
3 &16 &18 &20 &22 &24 &  &&   \\ \hline
4 &28 &30 &32 &  &&&&\\ \hline
\end{tabular}\bigskip

\noindent{}$n(a,b,7)$\\
\begin{tabular}{|r||r|r|r|r|r|}
\hline
$a\backslash b$ &1 &2 &3 &4 &5 \\ \hline\hline
1 &$\infty$ &\textbf{10} &\textbf{14} &\textbf{\color{blue}18} &\textbf{\color{blue}22} \\ \hline
2 &\textbf{14} &\textbf{19} &  & &  \\ \hline
\end{tabular}

\end{minipage}\hspace{1cm}\begin{minipage}[t]{4cm}

\noindent{}$n(a,b,8)$\\
\begin{tabular}{|r||r|r|r|r|}
\hline
$a\backslash b$ &1 &2 &3 &4 \\ \hline\hline
1 &$\infty$ &10 &\textbf{16} &\textbf{20} \\ \hline
2 &16 &24 & &  \\ \hline
\end{tabular}\vspace{1.3cm}

\noindent{}$n(a,b,9)$\\
\begin{tabular}{|r||r|r|r|}
\hline
$a\backslash b$ &1 &2 &3 \\ \hline\hline
1 &6 &\textbf{14} &\textbf{\color{blue}24} \\ \hline
2 &\textbf{24} &   &\\ \hline
\end{tabular}\vspace{1.3cm}

\noindent{}$n(a,b,10)$\\
\begin{tabular}{|r||r|r|r|}
\hline
$a\backslash b$ &1 &2 &3 \\ \hline\hline
1 &$\infty$ &16 &\textbf{28} \\ \hline
2 &\textbf{\color{blue}32} & &  \\ \hline
\end{tabular}\vfill{ }
\end{minipage}\bigskip 

Besides the values in the tables, we also computed $n(1,2,11)=24=M_{1}^{+}+4$, $n(1,2,12)=26=M_{2}^{+}$ and $n(5,2,6)=46=M_{2}^{+}$.

\section{General constructions for $(a,b,g)$-wcages}%
\label{sec:bounding}

Let $X$ be an $(r,g')$-cage. Assume that $X$ has an $a$-factor
$F$. Then $G=(F,X-F)$ is an $(a,r-a,g)$-wgraph for some girth $g$ with
$g'\leq g \leq 2g'$, thus $n(a,r-a,g)\leq n(r,g')$ in this
case. Assume further that $F$ is an $a$-factor of girth $g(F)\geq g'+1$, 
then we have $g\geq g'+1$: This is
true since a cycle of length $g'$ in $X$ can not be a cycle of $F$ and hence 
every cycle in $G$ must contain at least one heavy edge.
Moreover, if $X$ contains both a $a$-factor $F$  with $g(F)\geq g'+1$ and
a cycle $C$ of girth $g'$ with $|E(C)-E(F)|=1$, then $g=g'+1$.

A case of special interest is when $X$ is Hamiltonian. In this case, the 
Hamiltonian cycle $F$ is a $2$-factor an certainly $g(F)\geq g'+1$ whenever 
$r\geq 3$. It follows that $G=(F,X-F)$ is a $(2,r-2,g)$-wgraph for some $g$ 
with $g'+1\leq g\leq 2g'$.

These constructions can be applied in many cases to obtain upper
bounds for $(a,b,g)$-wcages. At least in the following cases, this method
matches the experimental results in the previous section and hence,
the produced wgraphs are indeed wcages:

\begin{enumerate}
\item Petersen's graph: $n(3,5)= 10$, gives $n(1,2,8)=10$. 
\item Heawood's graph: $n(3,6)=14$ gives $n(2,1,7)=14$ and $n(1,2,9)=14$.
\item McGee's graph: $n(3,7)=24$ gives $n(2,1,9)=24$ and $n(1,2,11)=24$.
\end{enumerate} 

Moreover, in the following cases the constructions give interesting $(a,b,g)$-wgraphs of small excess:

\begin{enumerate}
\item Hoffman-Singleton graph: $n(7,5)=50$ gives a $(5,2,6)$-wgraph on 50 vertices (but $n(5,2,6)=46$).
\item Tutte-Coxeter Graph: $n(3,8)=30$ gives a $(1,2,12)$-wgraph on 30 vertices, (but $n(1,2,12)=26$).
\end{enumerate}

Recall that a \emph{Moore cage} $X$ is an $(r,g')$-cage that attains the Moore bound, and that the \emph{Moore bound} is $n_{0}(r,g')=n_{0}(r,0,g')$ as described after Theorem~\ref{thm:MBounds}. Recall that this bounds come from the trees described in Section~\ref{sec:mbounds}, which in the case $a=r$, $b=0$ gives the standard Moore trees.

\begin{theorem} Assume $r\geq 3$, $g'\equiv 0$ and that there is an $(r,g')$-cage which is a Hamiltonian Moore cage, then we have that: 
$$n(2,r-2,g)\leq n_0(r,g')$$
for some $g$ with $g'+1\leq g\leq \frac{3}{2}g'-1$.
\end{theorem}
\begin{proof}
Let $X$ be an $(r,g')$-cage which is a Hamiltonian Moore cage, with $g'\equiv 0$ and $r\geq 3$. Let $C$ be a Hamiltonian cycle of $X$. Starting with any edge $xy\in X$, we can construct its Moore tree $T$ within $X$, which consist of $xy$ and two subtrees of $T$, $T_x$ and $T_y$, which are rooted at $x$ and $y$ respectively. Each of these trees have depth $\hh{g'-2}$. Since $X$ is a Moore cage, this tree $T$ is a spanning subgraph of $X$ and the rest of the edges of $X$ connect leaves in $T_x$ to leaves in $T_y$. 

Then we can construct the wgraph $G=(C,G-C)$, which is a $(2,r-2,g)$-wgraph, for some girth $g$ satisfying $g\geq g'+1$. We take an edge $xy \in C$ and construct the Moore tree $T$ in $X$ starting with it. Then $C$ must pass by $xy$ and go down from there to some leaf $\hat{x}$ of $T_x$ and some leaf $\hat{y}$ of $T_y$. Let $P_x$ and $P_y$ be the corresponding paths in $T$ that go from $x$ to $\hat{x}$ and from $y$ to $\hat{y}$.

Let $y_1, y_2, \ldots, y_{r-1}$ be the neighbours of $y$ in $T_y$, assume without loss that $y_1$ is in $C$. Consider the subtrees $T_{y_i}$ of $T$ rooted at $y_i$ for $i=1,2,\ldots,r-1$. All of these trees have height $\frac{g'-4}{2}$. Notice that $\hat{y}$ is a leaf of $T_{y_1}$. 

If $\hat{x}$ and $\hat{y}$ are adjacent in $X$, clearly $\hat{x}\hat{y}$ is not in $C$, and then we have a cycle of weight $g'+1$ on $G$. Suppose that $\hat{x}$ and $\hat{y}$ are not adjacent in $X$. Since the girth of $X$ is $g'$ and $\hat{x}$ is has degree $r$, $\hat{x}$ must be adjacent to exactly one leaf of each of $T_{y_1}, T_{y_2},\ldots ,T_{y_{r-1}}$. Hence, there is a neighbour, $y'$, of $\hat{x}$ among the leaves of $T_{y_1}$. Let $P'_{y_{1}}$ be the path in $T_1$ from $y_1$ to $y'$. 
Then there is a cycle $C'=P_x\cup xy \cup yy_{1}\cup P'_{y_{1}}\cup \hat{x}y'$ of length $g'$ in $X$ with at least $\hh{g'+2}$ edges in $C$ and at most $\hh{g'-2}$ not in $C$. Hence, the girth of $G$ is bounded by the weight of $C'$ as follows $g\leq \hh{g'+2} + 2\left(\hh{g'-2}\right)=\hh{3}g'-1.$
We conclude that $g'+1\leq g\leq \frac{3}{2}g'-1.$
\end{proof}

The previous theorem is still true if we replace $g'\equiv 0$ with $g'\equiv 1$ and the upper bound
with $g\leq \hh{3}g'-\hh{1}$. However, besides the hypothetical $(57,5)$-cage, the only Hamiltonian Moore cages of odd girth with $r\geq 3$ are the complete graphs and the Hoffman-Singleton graph \cite{EJ13}. The complete graphs only give easy bounds already established in Theorem~\ref{thm:mainfour}, and the Hoffman-Singleton graph gives a bad upper bound: $n(2,5,6)=16<50=n(7,5)$.

It is a well known observation that all Moore $(r,6)$-cages are incidence graphs of projective planes of order $(r-1)$ (see for instance \cite{CL96}). Also, we know from \cite{LMV13} that all of them are Hamiltonian. Furthermore, it is also well known that projective planes, and hence Moore cages of girth $6$, exists when $(r-1)$ is a prime power and that $n_0(r,6)=2(r^2-r+1)$.
Hence, the previous Theorem gives us the following corollary for girths $g\in\{7,8\}$:

\begin{corollary}\label{6hamiltcages}
Let $(r-1)$ be a prime power then:
$$n(2,r-2,g)\leq 2(r^2-r+1)$$
for some $g$ with $7\leq g \leq 8$.

Moreover, let $X$ be an $(r,6)$-cage. If $X$ has a $6$-cycle with all the edges on the Hamiltonian cycle except one, we have that: 
$$n(2,r-2,7)= 2(r^2-r+1),$$
otherwise: 
$$n(2,r-2,8)= 2(r^2-r+1).$$
\end{corollary}

We point out that it is a folklore conjecture that all cages are Hamiltonian except for the Petersen graph and hence these results should have wide applicability. For instance, besides the uses that we already mentioned above (Heawood, McGee), we can also apply them to the Tutte-Coxeter $(3,8)$-cage on $30$ vertices to obtain a $(2,1,9)$-wgraph of order $30$ (but $n(2,1,9)=24$). Also, Benson $(3,12)$-cage on $126$ vertices gives us a $(2,1,13)$-wgraph of order 126. This last is the best upper bound that we know for $n(2,1,13)$ and the Moore-like lower bound is $n_0(2,1,13)=66$, our algorithms can only provide the lower bound $n(2,1,13)\geq 68$. 

\bibliography{biblio}
\bibliographystyle{mapbib}

\end{document}